\documentclass{amsart}
\usepackage{amsmath}
  \usepackage{paralist}
  \usepackage{graphics} 
  \usepackage{epsfig} 
 \usepackage[colorlinks=true]{hyperref}
\hypersetup{urlcolor=blue, citecolor=red}

\usepackage{enumerate}
\usepackage{amsfonts}
\usepackage{amssymb}
\usepackage{amsthm}
\usepackage{mathrsfs}
\allowdisplaybreaks
\def \tilde {\widetilde}

\def \l {{\lambda}}

\def \R {{\mathbb {R}}}
\def \N {{\mathbb {N}}}

\def \phi {{\varphi}}

\newcommand{\1}{1\!\!\!\;\mathrm{l}}
\newcommand{\hs}[1]{\hskip -#1pt}

\newtheorem{theorem}{Theorem}[section]

\newtheorem{lemma}[theorem]{Lemma}
\newtheorem{proposition}[theorem]{Proposition}
\theoremstyle{definition}
\newtheorem{definition}[theorem]{Definition}
\newtheorem{remark}[theorem]{Remark}

\newtheorem{ipos}[theorem]{Hypotheses}

\title[Optimal H\"older regularity]{Optimal H\"older
regularity for nonautonomous Kolmogorov equations}
\author[L. Lorenzi]{}


 \email{luca.lorenzi@unipr.it}

\thanks{Work partially supported by the M.I.U.R. research project
Prin 2006 ``Kolmogorov equations''.}

\begin{document}
\maketitle

\centerline{\scshape Luca Lorenzi}
\medskip
{\footnotesize
 \centerline{Dipartimento di Matematica, Universit\`a degli Studi di Parma}
 \centerline{Viale G.P. Usberti 53/A, I-43124 Parma, Italy}
} 

\vskip .5truecm
\centerline{\dedicatory{\footnotesize To Claude-Michel
Brauner on the occasion of his 60th birthday}}
\bigskip


\begin{abstract}
We consider a class of nonautonomous elliptic operators ${\mathscr
A}$ with unbounded coefficients defined in  $[0,T]\times\R^N$ and we
prove optimal Schauder estimates for the solution to the parabolic
Cauchy problem $D_tu={\mathscr A}u+f$, $u(0,\cdot)=g$.
\end{abstract}

\section{Introduction}
\setcounter{section}{1} \setcounter{equation}{0}
\setcounter{theorem}{0}
In this paper we deal with a class of
nonautonomous elliptic operators ${\mathscr A}$ defined by
\begin{align*}
{\mathscr A}\varphi(t,x)&=
\sum_{i,j=1}^Nq_{ij}(t,x)D_{ij}\varphi(x)+\sum_{j=1}^Nb_j(t,x)D_j\varphi(x)+c(t,x)\varphi(x),
\end{align*}
on smooth functions $\varphi:\R^N\to\R$. We consider possibly
unbounded coefficients defined in $[0,T]\times\R^N$, smooth with
respect to $x$ and continuous or just measurable with respect to
time. We are interested in optimal Schauder estimates for the
solution to the Cauchy problem
\begin{equation}
\left\{
\begin{array}{lll}
D_tu(t,x)={\mathscr A}u(t,x)+g(t,x), & t\in [0,T], &x\in\R^N,\\[1.5mm]
u(0,x)=f(x), && x\in\R^N.
\end{array}
\right. \label{prob-intro}
\end{equation}

In the case when the coefficients of the operator ${\mathscr A}$ are
smooth enough in $[0,T]\times\R^N$ and satisfy suitable algebraic
and growth conditions at infinity (see Hypotheses \ref{ipos-1}), we
prove that, for any $f\in C^{2+\theta}_b(\R^N)$ and any continuous
function $g:[0,T]\times\R^N\to\R$ such that $g(t,\cdot)\in
C^{\theta}_b(\R^N)$ for any $t\in [0,T]$ and
\begin{equation}
\sup_{t\in [0,T]}\|g(t,\cdot)\|_{C^{\theta}_b(\R^N)}< +\infty,
\label{g-intro}
\end{equation}
there exists a unique classical bounded solution to problem
\eqref{prob-intro}, i.e., there exists a unique bounded function
$u:[0,T]\times\R^N\to\R$ which (i) is continuously differentiable in
$[0,T]\times\R^N$, once with respect to the time variable and twice
with respect to the spatial variables, (ii) solves the differential
equation in \eqref{prob-intro}, (iii) satisfies the condition
$u(0,\cdot)\equiv f$. Further, we have a Schauder type regularity result, i.e.,
the function $u(t,\cdot)$
belongs to $C^{2+\theta}_b(\R^N)$ for any $t\in [0,T]$ and there
exists a positive constant $C$ such that
\begin{equation}
\sup_{t\in [0,T]}\|u(t,\cdot)\|_{C^{2+\theta}_b(\R^N)}\le C\bigg (
\|f\|_{C^{2+\theta}_b(\R^N)}+\sup_{t\in
[0,T]}\|g(t,\cdot)\|_{C^{\theta}_b(\R^N)}\bigg ).
\label{stima-intro}
\end{equation}

In the case of discontinuous coefficients, under suitable
assumptions (see Hypotheses \ref{ipos-2}), we can still prove an
existence and uniqueness theorem for problem \eqref{prob-intro} as
well as optimal Schauder estimates. Here, we assume that the
function $g$ is measurable in $[0,T]\times\R^N$ and it satisfies
condition \eqref{g-intro}.
The lack of regularity of the data with respect to the time variable
prevents the solution from being continuously differentiable with respect to
the time variable. Hence, we introduce an appropriate definition of
solution to problem \eqref{prob-intro}, adapted to the discontinuity
of the data (see Definition \ref{defin-1}). Then, we prove that there exists a unique
solution to the problem \eqref{prob-intro} in the sense of
Definition \ref{defin-1}. This solution $u$ satisfies estimate
\eqref{stima-intro}. Roughly speaking, the main
difference with the case when the coefficients are smooth, is that
now the function $u(\cdot,x)$ is differentiable for any $x\in\R^N$ almost everywhere in
$[0,T]$ and that the differential equation is
satisfied almost everywhere in $[0,T]\times\R^N$. Both for smooth and nonsmooth
coefficients, the uniqueness of
the solution to problem \eqref{prob-intro} follows from a variant of
the classical maximum principle (see Propositions \ref{maximum} and
\ref{maxprinc}), which can be proved assuming the existence of a
suitable Lyapunov function (see Hypothesis \ref{ipos-1}(vii)).

For autonomous equations with unbounded smooth coefficients,
Schauder theorems of this type were obtained in
\cite{BL,lunardi-studia} as a consequence of optimal estimates in
the sup norm of the spatial derivatives of the solution to the
homogeneous Cauchy problem
\begin{equation}
\left\{
\begin{array}{lll}
D_tu(t,x)={\mathscr A}u(t,x), &t\in [0,T], &x\in\R^N,\\[1.5mm]
u(0,x)=f(x), &&x\in\R^N,
\end{array}
\right. \label{pb-omog-intro}
\end{equation}
when $f$ belongs to suitable spaces of H\"older continuous
functions. See also \cite{cerrai}.

To the best of our knowledge, the first papers dealing with optimal
Schauder estimates of the type \eqref{stima-intro} are
\cite{KCL,KCL-1,KCL-2} where bounded and continuous coefficients
satisfying \eqref{g-intro} were considered. Recently, the results in
\cite{KCL,KCL-1,KCL-2} have been extended in \cite{Lo1} to
discontinuous bounded coefficients still satisfying \eqref{g-intro}
and in \cite{Lo2} to operators ${\mathscr A}$ of the type
\begin{align*}
{\mathscr A}\varphi(t,x)=
\sum_{i,j=1}^Nq_{ij}(t)D_{ij}\varphi(x)+\sum_{i,j=1}^Nb_{ij}(t)x_jD_i\varphi(x),
\end{align*}
with bounded coefficients $q_{ij}$ and $b_{ij}$ ($i,j=1,\ldots,N$). When ${\mathscr A}$ is an Ornstein-Uhlenbeck
operator things are easier than in the general case, since an
explicit formula for the solution to \eqref{pb-omog-intro} is known.
Hence, one can obtain uniform estimates for the spatial derivatives of the solution to problem
\eqref{pb-omog-intro} just differentiating the formula which defines
$u$. On the contrary, no explicit formulas are available for more
general elliptic operators.

Very recently, Krylov and Priola (see \cite{KP}) have studied more
general nonautonomous elliptic operators with unbounded and less
regular coefficients, using different techniques. Their interesting
results show global Schauder estimates for the solutions to the
parabolic equation $D_tu+{\mathscr A}u=f$ in $(T,+\infty)\times\R^N$
when $T\in [-\infty,+\infty)$ and, as a byproduct, Schauder
estimates for the solution to problem \eqref{prob-intro}, thus
extending the results of \cite{Lo1}. Roughly speaking, in \cite{KP}
the diffusion coefficients are supposed to be bounded, whereas the
drift coefficients may grow at most linearly at infinity, with
respect to the spatial variables.

In this paper we prove optimal Schauder estimates for nonautonomous
elliptic operators whose coefficients may grow faster than linearly
at infinity. In the first part of the paper (see Section
\ref{sect-2}), we consider the case when the coefficients are smooth
in $[0,T]\times\R^N$, adapting the techniques in
\cite{BL,lunardi-studia}. First, in Subsection \ref{subsect-2.1}, we
prove that problem \eqref{pb-omog-intro} admits a unique bounded classical
solution $u$ for any $f\in C_b(\R^N)$. Subsection \ref{subsect-2.2}
is then devoted to prove uniform estimates for the derivatives (up
to third-order, and with respect to the sup-norm in $\R^N$) of the
solution to problem \eqref{pb-omog-intro} when $f$ belongs to
suitable spaces of H\"older continuous functions. Finally, these
uniform estimates and an abstract interpolation method (see
\cite{lunardi-sem}) yield optimal Schauder estimates (see Subsection
\ref{subsect-2.3}). In the second part of the paper, we turn our
attention to the case of discontinuous (in time) coefficients. We
prove the optimal Schauder estimates approximating the operator
${\mathscr A}$ by a sequence of elliptic operators ${\mathscr
A}^{(n)}$ which satisfy the assumptions of the first part of the
paper, and using a compactness argument. The arguments in the proof
of Theorem \ref{thm-main-2} can then be used to weaken a bit the
assumptions of Section \ref{sect-2} and prove the Schauder estimates
of Theorem \ref{thm-schauder} without any assumption of H\"older in
time regularity of the coefficients of ${\mathscr A}$. See Theorem
\ref{thm-schauder-cont-2}. Finally, in Section \ref{sect-4}, we
exhibit a class of elliptic operators to which the optimal Schauder
estimates may be applied.

\subsection*{Notations}
$C_b(\R^N)$ denotes the set of all bounded and continuous functions
$f:\R^N\to\R$. We endow it with the sup-norm $\|\cdot\|_{\infty}$.
For any $k>0$ (possibly $k=+\infty$), $C^k_b(\R^N)$ denotes the
subset of $C_b(\R^N)$ of all functions $f:\R^N\to\R$ that are
continuously differentiable in $\R^N$ up to $[k]$th-order, with
bounded derivatives and such that the $[k]$th-order derivatives are
$(k-[k])$-H\"older continuous in $\R^N$. $C^k_b(\R^N)$ is endowed
with the norm $\|f\|_{C^k_b(\R^N)}:=\sum_{|\alpha|\le
[k]}\|D^{\alpha}f\|_{\infty}
+\sum_{|\alpha|=[k]}[D^{\alpha}f]_{C^{k-[k]}_b(\R^N)}$.
$C^k_c(\R^N)$ ($k\in\mathbb N\cup\{+\infty\}$)  denotes the subset
of $C^k_b(\R^N)$ of all compactly supported functions.

For any domain $D\subset\R\times\R^N$ and any $\alpha\in (0,1)$,
$C^{\alpha/2,\alpha}(D)$ denotes the space of all
H\"older-continuous functions with respect to the parabolic distance
of $\R^{N+1}$. Similarly, for any $h,k\in\N\cup\{0\}$ and any
$\alpha\in [0,1)$, $C^{h+\alpha/2,k+\alpha}(D)$ denotes the set of
all functions $f:D\to\R$ which (i) are continuously differentiable
in $D$ up to the $h$th-order with respect to time variable, and up
to the $k$th order with respect to the spatial variables, (ii) the
derivatives of maximum order are in $C^{\alpha/2,\alpha}(D)$ (here,
$C^{0,0}:=C$). Finally, we use the notation
$C^{h+\alpha/2,k+\alpha}_{\rm loc}(D)$ to denote the set of all
functions $f:D\to\R$ which are in $C^{h+\alpha/2,k+\alpha}(D_0)$ for
any compact set $D_0\subset D$.

For any measurable set $E$, we denote by $\1_E$ the characteristic
function of $E$, i.e., $\1_E(x)=1$ if $x\in E$, $\1_E(x)=0$
otherwise.

Given a $N\times N$ matrix we denote by ${\rm Tr}(Q)$ its trace.
Further, we denote by $\langle \cdot,\cdot\rangle$ the Euclidean
inner product of $\R^N$.

\section{The case of smooth coefficients}
\label{sect-2} \setcounter{equation}{0} Throughout this section, we
make the following assumptions on the coefficients $q_{ij}$, $b_j$
($i,j=1,\ldots,N$) and $c$ of the operator ${\mathscr A}$. We denote
by $Q(t,x)$ and $b(t,x)$ the matrix whose entries are the coefficients
$q_{ij}(t,x)$, and the vector whose entries are the coefficients
$b_j(t,x)$, respectively.
\begin{ipos}
\label{ipos-1} ~
\par
\noindent
\begin{enumerate}[\rm (i)]
\item
the coefficients $q_{ij},b_j$ $(i,j=1,\ldots,N$) and $c$ are thrice
continuously differentiable with respect to the spatial variables in
$[0,T]\times\R^N$ and they belong to
$C^{\delta/2,\delta}([0,T]\times B(0,R))$ for some $\delta\in (0,1)$
and any $R>0$, together with their first-, second- and third-order
spatial derivatives;
\item
$q_{ij}(t,x)=q_{ji}(t,x)$ for any $i,j=1,\ldots, N$ and any
$(t,x)\in [0,T]\times\R^N$, and
\begin{eqnarray*}
\langle Q(t,x)\xi,\xi\rangle\ge\nu(t,x)|\xi|^2,\qquad\;\,t\in
[0,T],\;\,\xi,x\in\R^N,
\end{eqnarray*}
for some function $\nu:(0,T)\times\R^N\to\R$ such that
\begin{eqnarray*}
\inf_{(t,x)\in (0,T)\times\R^N}\nu(t,x)=\nu_0>0; \end{eqnarray*}
\item
there exist positive constants $C_1,C_2,C_3$ such that
\begin{align}
&|Q(t,x)x|\le C_1(1+|x|^2)\nu(t,x),\label{cond:qij-1}\\
&{\rm Tr}(Q(t,x))\le C_2(1+|x|^2)\nu(t,x), \label{cond:qij-2}\\
&\langle b(t,x),x\rangle \le C_3(1+|x|^2)\nu(t,x), \label{cond:b}
\end{align}
for any $t\in [0,T]$ and any $x\in\R^N$;
\item
$c(t,x)\le c_0$ for some real constant $c_0$ and any $(t,x)\in
[0,T]\times\R^N$;
\item
there exist three positive constants $K_1$, $K_2$ and $K_3$ such
that
\begin{align*}
&|D^{\beta}q_{ij}(t,x)|\le K_{|\beta|}\nu(t,x),\\
&\sum_{h,k,l,m=1}^ND_{lm}q_{hk}(t,x)\xi_{hk}\xi_{lm}\le
K_2\nu(t,x)\sum_{h,k=1}^N \xi_{hk}^2,
\end{align*}
for any $i,j=1,\ldots,N$, any $|\beta|=1,3$, any $N\times N$ symmetric matrix $\Xi=(\xi_{hk})$ and any $(t,x)\in
[0,T]\times\R^N$;
\item
there exist three functions $d,r,\varrho:[0,T]\times\R^N\to\R$, with $\varrho\ge \varrho_0$ for some positive constant $\varrho_0$,
and positive constants $L_1,L_2,L_3$ such that
\begin{align}
\label{dissip}
&\langle Db(t,x)\xi,\xi\rangle\le d(t,x)|\xi|^2,\\
\label{cond-deriv-b}
&|D^\beta b_j (t,x)|\le r(t,x),\\
\label{cond-deriv-c}
&|D^{\gamma}c(t,x)|\le\varrho(t,x),\\
\label{cond-funz}
&d(t,x)+L_1r(t,x)+L_2\varrho^2(t,x)\le L_3\nu(t,x),
\end{align}
for any $t\in [0,T]$, any $|\beta|=2,3$, any $|\gamma|=1,2,3$, any $j=1,\ldots,N$ and
any $x,\xi\in\R^N$, where $Db=(D_jb_i)$;
\item
there exist a positive function $\varphi:\R^N\to\R$ and $\lambda>0$
such that $\varphi$ tends to $+\infty$ as $|x|\to +\infty$ and
\begin{equation}
\sup_{(t,x)\in [0,T]\times\R^N}\{{\mathscr
A}\varphi(t,x)-\lambda\varphi(x)\}<+\infty. \label{liapunov}
\end{equation}
\end{enumerate}
\end{ipos}

\subsection{The homogeneous Cauchy problem associated with the
operator ${\mathscr A}$} \label{subsect-2.1}
In this subsection, for
any $s\in [0,T)$, we consider the homogeneous Cauchy problem
\begin{equation}
\left\{
\begin{array}{lll}
D_tu(t,x)={\mathscr A}u(t,x), & t\in (s,T], & x\in \R^N,\\[1.5mm]
u(s,x)=f(x), && x\in \R^N.
\end{array}
\right. \label{pb-omog}
\end{equation}
We are going to prove that, for any $f\in C_b(\R^N)$, problem
\eqref{pb-omog} admits a unique bounded classical solution $u$
(i.e., there exists a unique bounded and continuous function
$u:[s,T]\times\R^N\to\R$ which is continuously differentiable in
$(s,T]\times\R^N$, once with respect to time and twice with respect
to the spatial variables, that satisfies \eqref{pb-omog}).

Uniqueness of the bounded classical solution to problem \eqref{pb-omog}
follows from the following variant of the classical maximum
principle.

\begin{proposition}\label{maximum}
Fix $s\in [0,T)$. If $u \in C_b([s,T]\times\R^N)\cap C^{1,2}((s,T] \times\R^N)$
satisfies
\begin{eqnarray*}
\left\{
\begin{array}{lll}
D_tu(t,x) - {\mathscr A}u(t,x) \le  0, & t\in
(s,T], &x\in\R^N,\\[1.5mm]
u(s,x) \le 0, && x \in \R^N,
\end{array}\right.
\end{eqnarray*}
then $u \leq 0$.
\end{proposition}

\begin{proof}
The proof can be obtained repeating the arguments in the proof of
the forthcoming Proposition \ref{maxprinc}, using the classical
maximum principle instead of the Nazarov-Ural'tseva maximum
principle.
\end{proof}

\begin{theorem}
\label{thm-3.2}
Under Hypotheses $\ref{ipos-1}$, for any $f\in C_b(\R^N)$, the Cauchy problem
\eqref{pb-omog} admits a unique bounded classical solution $u$. Moreover,
\begin{equation}
\|u(t,\cdot)\|_{\infty}\le e^{c_0(t-s)}\|f\|_{\infty},\qquad\;\,t\in
[s,T], \label{estim-u-om}
\end{equation}
where $c_0$ is as in Hypothesis $\ref{ipos-1}(iv)$.
\end{theorem}

\begin{proof}
As it has been already remarked, the uniqueness part is a
straightforward consequence of Proposition \ref{maximum}.

To prove the existence part of the statement, we first consider the
case when $f$ is nonnegative. For any $n\in\N$, let us consider the
Cauchy-Dirichlet problem
\begin{equation}
\left\{
\begin{array}{lll}
D_tv(t,x)={\mathscr A}v(t,x), & t\in (s,T], & x\in B(0,n),\\[1.5mm]
v(t,x)=0, & t\in (s,T], & x\in \partial B(0,n),\\[1.5mm]
v(s,x)=f(x), && x\in B(0,n).
\end{array}
\right. \label{pb-omog-approx}
\end{equation}
By classical results (see e.g., \cite[Theorem 3.5]{friedman}) this
problem admits a unique bounded solution $u_n\in C^{1,2}((s,T]\times
B(0,n))$ which is continuous in $[s,T]\times
\overline{B(0,n)}\setminus\{(s,x): x\in\partial B(0,n)\}$. Moreover,
the classical maximum principle shows that
\begin{equation}
|u_n(t,x)|\le e^{c_0(t-s)}\|f\|_{\infty},\qquad\;\,(t,x)\in
[s,T]\times \overline{B(0,n)} \label{stima-approx}
\end{equation}
and that the sequence $u_n(t,x)$ is increasing for any fixed
$(t,x)$. Classical interior Schauder
estimates imply that the sequence $(u_n)$ is bounded in
$C^{1+\delta/2,2+\delta}(D)$ for any compact set $D\subset
(s,T]\times \R^N$. Here, $\delta$ is the same number as in
Hypothesis \ref{ipos-1}(i). The Ascoli-Arzel\`a theorem implies that
$u_n$ converges in $C^{1,2}(D)$, for any $D$ as above, to a function
$u$ which belongs to $C^{1+\delta/2,2+\delta}_{\rm
loc}((s,T)\times\R^N)$ and solves the differential equation in
\eqref{pb-omog-approx}.

Showing that $u$ is continuous up to $t=s$ and $u(s,\cdot)=f$ is a
bit more tricky. It is straightforward if $f\in
C^{2+\delta}_c(\R^N)$ since, in this case, the classical Schauder
estimates show that $(u_n)$ is bounded in
$C^{1+\delta/2,2+\delta}(D)$ for any compact set $D\subset
[s,T]\times\R^N$. Hence, $u_n$ converges to $u$ uniformly in
$[s,T]\times K$ for any compact set $K\subset\R^N$, so that $u$ is
continuous up to $t=s$ and it therein equals the function $f$.
Taking the limit as $n\to +\infty$ in \eqref{stima-approx}, estimate
\eqref{estim-u-om} follows immediately. Using this estimate, it is
then easy to show that $u$ is continuous up to $t=s$ and therein
equals the function $f$ also in the case when $f\in C_c(\R^N)$.

In the general case when $f\in C_b(\R^N)$, continuity up to $t=s$ can be obtained by a
localization argument. Let us fix $x_0\in\R^N$ and let
$\eta:\R^N\to\R$ be a continuous function such that
$\1_{B(x_0,1)}\le\eta\le\1_{B(x_0,2)}$. Writing $f=f\eta+f(1-\eta)$,
we split $u_n$ into the sum of the functions $v_n$ and $w_n$
which are, respectively, the solutions to \eqref{pb-omog-approx}
with initial data $\eta f$ and $(1-\eta)f$. Since $\eta f$ is compactly
supported in $\R^N$, $v_n$ converges to the solution
$v$ to problem \eqref{pb-omog} with $f$ being replaced by $\eta f$.
On the other hand, the classical maximum principle shows that
\begin{eqnarray*}
|w_n(t,x)|\le K(1-z_n(t,x)),\qquad\;\,(t,x)\in
[s,T]\times\overline{B(0,n)},
\end{eqnarray*}
for any $n\in\N$. Here, $K=\|f\|_{\infty}$ and $z_n$ denotes the solution to problem \eqref{pb-omog-approx}, with $f$
being replaced with the function $\eta$. Since the function
$v_n+w_n$ converges to $u$ and $z_n$ converges to
the solution  to problem \eqref{pb-omog} as $n\to +\infty$, it follows that
\begin{eqnarray*}
|u(t,x)-f(x)|\le |v(t,x)-f(x)|+K(1-z(t,x)).
\end{eqnarray*}
Letting $(t,x)\to (s,x_0)$, one easily obtains that $u(t,x)-f(x)$
tends to $0$. Hence, $u$ is continuous at $(t,x)=(s,x_0)$, where it
equals $f(x_0)$. This completes the proof for nonnegative data $f$.

For a general $f\in C_b(\R^N)$, the proof can be obtained splitting
$f=f^+-f^-$, where $f^+=\max\{f,0\}$, $f^{-}=(-f)^+$. Clearly, the
solution to problem \eqref{pb-omog} will be given by $u_+-u_-$,
where $u_+$ and $u_+$ are the solutions to problem \eqref{pb-omog},
with $f$ being replaced, respectively, by $f^+$ and $f^-$.
\end{proof}

In the rest of the paper, for any $f\in C_b(\R^N)$, we denote by
$G(t,s)f$ the value at time $t$ of the unique bounded classical solution $u$ to
problem \eqref{pb-omog}.

\subsection{Uniform estimates} \label{subsect-2.2}
This subsection is devoted to the proof of the following theorem.

\begin{theorem}\label{C0-C3}
Let Hypotheses $\ref{ipos-1}$ be satisfied.
Then, for any $\alpha,\beta\in [0,3]$, with $\alpha\le\beta$, there exists a positive constant
$C=C(\alpha,\beta)$ such that
\begin{equation}
\|G(t,s)f\|_{C^{\beta}_b(\R^N)}\le C(t-s)^{-\frac{\beta-\alpha}{2}}\|f\|_{C^{\alpha}_b(\R^N)},
\qquad\;\,f\in C_b^{\alpha}(\R^N),
\label{stimasem}
\end{equation}
for any $t\in (s,T]$.
\end{theorem}

The proof will be obtained in two steps. In the first one we will
prove \eqref{stimasem} when $\alpha,\beta\in\N$. Then, using an
interpolation argument we extend \eqref{stimasem} to any
$\alpha,\beta$ as in the statement of the theorem.

Let us introduce a few more notation. We denote by ${\mathscr Q}$
and ${\mathscr B}$ the operators defined on functions $f,g\in
C^{0,1}([0,T]\times\R^N)$ by
\begin{eqnarray*}
{\mathscr Q}(f,g)=\langle Q\nabla_xf,\nabla_xg\rangle,\qquad
{\mathscr B}(f,g)=\langle Db\nabla_xf,\nabla_xg\rangle.
\end{eqnarray*}

\subsubsection{\bf The case when $\alpha,\beta\in\N$} \label{subsect-3.1}
We first consider the case when $\alpha=0$, $\beta=3$. For any
$n\in\N$, let $\eta:\R^N\to\R$ be the radial function defined by
$\eta(x)=\psi(|x|/n)$ for any $x\in\R^N$, where $\psi$ is a smooth
nonincreasing function such that $\1_{[0,1/2]}\le
\psi\le\1_{[0,1]}$. We fix $s\in (0,T)$, and define the
function
\begin{eqnarray*}
v_n(t,x)&\hs{5}=\hs{5}&|u_n(t,x)|^2+a(t-s)\eta^2|\nabla_xu_n(t,x)|^2+a^2(t-s)^2\eta^4|D^2_xu_n(t,x)|^2\\
&\hs{5}&+a^3(t-s)^3\eta^6|D^3_xu_n(t,x)|^2,
\end{eqnarray*}
for any $t\in (s,T]$ and any $x\in B(0,n)$, where $u_n$ is the
(unique) classical solution of the Dirichlet Cauchy problem
\begin{eqnarray*}
\left\{
\begin{array}{lll}
D_tu(t,x)={\mathscr A}u(t,x), & t\in [s,T], & x\in B(0,n),\\[1.5mm]
u(t,x)=0, & t\in [s,T], & x\in \partial B(0,n),\\[1.5mm]
u(s,x)=\eta(x)f(x), && x\in\overline{B(0,n)}.
\end{array}
\right.
\end{eqnarray*}

By classical results (see e.g., \cite{friedman}), the function $v_n$
belongs to $C^{1,2}((s,T)\times B(0,n))$. Moreover, it can be
extended by continuity up to $t=s$ setting $v_n(s,\cdot)=|\eta
f|^2$.

A long but straightforward computation shows that the function $v_n$
solves the Cauchy problem
\begin{eqnarray*}
\left\{
\begin{array}{lll}
D_tv_n(t,x)=\mathcal{A}v_n(t,x)+g_n(t,x),\quad &t\in [s,T], &x\in B(0,n),\\[1.5mm]
v_n(t,x)=0, &t\in [s,T], &x\in\partial B(0,n),\\[1.5mm]
v_n(s,x)=(\eta f)^2(x), &&x\in\overline{B(0,n)},
\end{array}
\right.
\end{eqnarray*}
where $g_n=\sum_{i=1}^9g_{i,n}$, with
\begin{align*}
g_{1,n}=&-2{\mathscr Q}(u_n,u_n)
\displaystyle-2a(t-s)\eta^2\sum_{i=1}^N{\mathscr Q}(D_iu_n,D_iu_n)\nonumber\\
&-2a^2 (t-s)^2\eta^4\sum_{i,j=1}^N{\mathscr Q}(D_{ij}u_n,D_{ij}u_n)\nonumber\\
&-2a^3(t-s)^3\eta^6\sum_{i,j,h=1}^N {\mathscr Q}
(D_{ijh}u_n,D_{ijh}u_n),
\end{align*}
\begin{align*}
g_{2,n}=& 2a(t-s)\eta^2{\mathscr B}(u_n,u_n)
+4a^2 (t-s)^2\eta^4\sum_{i=1}^N{\mathscr B}(D_iu_n,D_iu_n)\nonumber\\
& +6a^3(t-s)^3\eta^6\sum_{i,j=1}^N{\mathscr B}(D_{ij}u_n,D_{ij}u_n),
\end{align*}
\[
g_{3,n}=-2a(t-s)\big (|\nabla_xu_n|^2+6a
(t-s)\eta^2|D^2_xu_n|^2+15a^2(t-s)^2\eta^4|D^3_xu_n|^2\big
){\mathscr Q}(\eta,\eta),
\]
\begin{align*}
g_{4,n}=& -2({\mathscr A}\eta-c\eta)\big \{a(t-s)\eta|\nabla_xu_n|^2
+2a^2(t-s)^2\eta^3|D^2_xu_n|^2\nonumber\\
&\qquad\qquad\qquad\quad+3a^3(t-s)^3\eta^5|D^3_xu_n|^2\big \}\nonumber\\
&-8a(t-s)\eta\sum_{i=1}^N{\mathscr Q}(\eta,D_iu_n)D_iu_n
-16a^2 (t-s)^2\eta^3\sum_{i,j=1}^N{\mathscr Q}(\eta,D_{ij}u_n)D_{ij}u_n\nonumber\\
&-24a^3(t-s)^3\eta^5\sum_{i,j,h=1}^N{\mathscr
Q}(\eta,D_{ijh}u_n)D_{ijh}u_n,
\end{align*}
\begin{align*}
g_{5,n}=& 2a(t-s)\eta^2\sum_{i,j,h=1}^N D_hq_{ij}D_hu_nD_{ij}u_n\nonumber\\
&+4a^2 (t-s)^2\eta^4\sum_{i,j,h,k=1}^N D_hq_{ij}D_{hk}u_nD_{ijk}u_n\nonumber\\
& +6a^3(t-s)^3\eta^6\sum_{i,j,h,k,l=1}^ND_hq_{ij}D_{hkl}u_nD_{ijkl}u_n,
\end{align*}
\begin{align*}
g_{6,n}=& 2a^2 (t-s)^2\eta^4\sum_{i,j,h,k=1}^N D_{hk}q_{ij}D_{ij}u_nD_{hk}u_n\nonumber\\
&+6a^3(t-s)^3\eta^6\sum_{i,j,h,k,l=1}^N D_{hk}q_{ij}D_{ijl}u_nD_{hkl}u_n\nonumber\\
& +2a^2 (t-s)^2\eta^4\sum_{i,j,h=1}^N D_{jh}b_iD_iu_nD_{jh}u_n\nonumber\\
&+6a^3(t-s)^3\eta^6\sum_{i,j,h,k=1}^N D_{jh}b_iD_{ik}u_nD_{jhk}u_n,
\end{align*}
\begin{align*}
g_{7,n}=\,&2a^3(t-s)^3\eta^6\sum_{i,j,h,k,l=1}^N
D_{hkl}q_{ij}D_{ij}u_nD_{hkl}u_n\\
& +2a^3(t-s)^3\eta^6\sum_{i,j,h,k=1}^N D_{jhk}b_iD_iu_nD_{jhk}u_n,
\end{align*}
\begin{eqnarray*}
g_{8,n}=a\eta^2|\nabla_xu_n|^2+2a^2(t-s)\eta^4|D^2_xu_n|^2+3a^3(t-s)^2\eta^6|D^3_xu_n|^2,
\end{eqnarray*}
\begin{align*}
g_{9,n}=\,&2cv_n+2a(t-s)\eta^2 u_n\langle
\nabla_xc,\nabla_xu_n\rangle
+4a^2(t-s)^2\eta^4\langle D^2_xu_n\nabla_x c,\nabla_xu_n\rangle\nonumber\\
&+2a^2(t-s)^2\eta^4u_n{\rm Tr}(D^2_xc\,D^2_xu_n)
+2a^3(t-s)^3\eta^6u_n\sum_{i,j,h=1}^ND_{ijh}cD_{ijh}u_n\nonumber\\
&+6a^3(t-s)^3\eta^6\bigg (\sum_{i,j,h=1}^ND_icD_{jh}u_nD_{ijh}u_n+
\sum_{i,j,h=1}^ND_{jh}cD_iu_nD_{ijh}u_n\bigg ).
\end{align*}
Taking Hypothesis \ref{ipos-1}(ii) into account, we easily deduce
that
\begin{align}
g_{1,n}\le& -2\nu |\nabla_xu_n|^2-2a(t-s)\eta^2\nu|D^2_xu_n|^2-2a^2(t-s)^2\eta^4\nu|D^3_xu_n|^2\nonumber\\
&-2a^3(t-s)^3\eta^6\nu|D^4_xu_n|^2.
\label{stimasug1}
\end{align}

As far as the function $g_{2,n}$ is concerned, we observe that
condition \eqref{dissip} implies that ${\mathscr B}(\zeta,\zeta)\le
d|\nabla\zeta|^2$ for any $\zeta\in C^1(\R^N)$. Hence, we can
estimate
\begin{equation}
g_{2,n}\le 2a(t-s)d\eta^2
|\nabla_xu_n|^2+4a^2(t-s)^2d\eta^4|D^2_xu_n|^2+6a^3(t-s)^3d\eta^6|D^3_xu_n|^2.
\label{stimasug4}
\end{equation}

The function $g_{3,n}$ can be estimated trivially from above by
zero. So, let us consider the function $g_{4,n}$. Using conditions
\eqref{cond:qij-1}, \eqref{cond:qij-2} and \eqref{cond:b} and
recalling that $\nabla\eta$ and $D^2\eta$ identically vanish in
$B(0,n/2)$ and in $\R^N\setminus B(0,n)$, it is not difficult to
check that
\begin{align*}
&|{\rm Tr}(Q(t,x)D^2_x\eta(x))|\le \frac{1+n^2}{n^3}\left
(2C_1\|\psi''\|_{\infty}
+2C_2\|\psi'\|_{\infty}n+2C_1\|\psi'\|_{\infty}\right
)\nu(t,x),\\
&\langle b(t,x),\nabla_x\eta\rangle\ge -
2C_3\|\psi'\|_{\infty}\frac{1+n^2}{n^2}\nu(t,x),
\end{align*}
for any $(t,x)\in [0,T]\times\R^N$. Hence, for $n$ sufficiently
large, it holds that
\begin{equation}
{\mathscr A}\eta-c\eta\ge -C'\nu:=-4\left
((C_1+C_2+C_3)\|\psi'\|_{\infty}+C_1\|\psi''\|_{\infty}\right )\nu.
\label{1}
\end{equation}
Arguing similarly, we can estimate
\begin{equation}
|{\mathscr Q}(\eta,\zeta)|\le 4C_1\|\psi'\|_{\infty}\nu
|\nabla_x\zeta|:=C''\nu |\nabla_x\zeta|, \label{2}
\end{equation}
for any function $\zeta\in C^1(\R^N)$. Using \eqref{1} and
\eqref{2}, we now get easily that
\begin{align*}
g_{4,n}\le & 2aC'(t-s)\nu\eta
|\nabla_xu_n|^2+4a^2C'(t-s)^2\nu\eta^3 |D^2_xu_n|^2\nonumber\\
&+6a^3C'(t-s)^3\nu\eta^5|D^3_xu_n|^2+8aC''(t-s)\eta\nu|\nabla_xu_n||D^2_xu_n|\nonumber\\
&+16a^2C''(t-s)^2\eta^3\nu|D^2_xu_n||D^3_xu_n|\nonumber+24a^3C''(t-s)^3\eta^5\nu|D^3_xu_n||D^4_xu_n|.
\end{align*}
Using Young inequality we can estimate
\begin{align*}
&\eta|\nabla_xu_n||D^2_xu_n|\le \varepsilon\eta^2
|D^2_xu_n|^2+\frac{1}{4\varepsilon}|\nabla_xu_n|^2,\\[1.5mm]
&\eta^3|D^2_xu_n||D^3_xu_n|\le \varepsilon
\eta^4|D^3_xu_n|^2+\frac{1}{4\varepsilon}\eta^2|D^2_xu_n|^2,\\[1.5mm]
&\eta^5|D^3_xu_n||D^4_xu_n|\le \varepsilon
\eta^6|D^4_xu_n|^2+\frac{1}{4\varepsilon}\eta^4|D^3_xu_n|^2,
\end{align*}
for any $\varepsilon>0$. Hence,
\begin{align}
g_{4,n}\le & 2a\left (C'+\frac{C''}{\varepsilon}\right )T\nu|\nabla_xu_n|^2\nonumber\\
&+4a\left (aC'T+2C''\varepsilon+\frac{C''}{\varepsilon}aT\right )(t-s)\nu\eta^2|D^2_xu_n|^2\nonumber\\
&+ 2a^2\left (3aC'T+8\varepsilon C''+3\frac{C''}{\varepsilon}aT\right )(t-s)^2\nu\eta^4|D^3_xu_n|^2\nonumber\\
&+24a^3C''\varepsilon(t-s)^3\nu\eta^6|D^4_xu_n|^2. \label{stimasug3}
\end{align}

The terms $g_{5,n}$, $g_{6,n}$ and $g_{7,n}$ can be estimated in a
similar way. Hypotheses \ref{ipos-1}(v) and \ref{ipos-1}(vi) imply
that
\begin{align*}
&\Bigg |\sum_{i,j,h=1}^ND_hq_{ij}D_{ij}\zeta D_h\zeta\Bigg |\le
NK_1\nu|\nabla\zeta||D^2\zeta|,\\[1.5mm]
&\sum_{i,j,h,k=1}^ND_{hk}q_{ij}D_{ij}\zeta D_{hk}\zeta\le
K_2\nu|D^2\zeta|^2,\\[1.5mm]
&\Bigg |\sum_{i,j,h,k,l=1}^ND_{hkl}q_{ij}D_{ij}\zeta
D_{hkl}\zeta\Bigg |\le
NK_3\nu|D^2\zeta||D^3\zeta|,\\[1.5mm]
&\Bigg |\sum_{j,h,k=1}^ND_{hk}b_jD_j\zeta D_{hk}\zeta\Bigg |\le
rN^{\frac{1}{2}}|\nabla\zeta||D^2\zeta|,\\[1.5mm]
&\Bigg |\sum_{j,h,k,l=1}^ND_{hkl}b_jD_j\zeta D_{hkl}\zeta\Bigg |\le
rN^{\frac{1}{2}}|\nabla\zeta||D^3\zeta|,
\end{align*}
for any smooth function $\zeta$. Hence,
\begin{align}
g_{5,n}\le &aTK_1\frac{N^2}{2\varepsilon}\nu|\nabla_xu_n|^2+
aK_1\left (2\varepsilon
+aT\frac{N^2}{\varepsilon}\right )(t-s)\nu\eta^2|D^2_xu_n|^2\nonumber\\
&+a^2K_1\left (4\varepsilon +3aT\frac{N^2}{2\varepsilon}\right
)(t-s)^2\nu\eta^4 |D^3_xu_n|^2+6a^3(t-s)^3\varepsilon K_1\nu\eta^6
|D^4_xu_n|^2, \label{stimasug5}
\end{align}
\begin{align}
g_{6,n}\le &a^2(t-s)\eta^2r\frac{N}{2\varepsilon}T|\nabla_xu_n|^2
+a^2(t-s)^2\eta^4\left [2K_2\nu\hs{1}+\hs{2}\left (2\varepsilon
+3a\frac{N}{2\varepsilon}T\right )r\right ]
|D^2_xu_n|^2\nonumber\\
&+6a^3(t-s)^3\eta^6\left (K_2\nu+\varepsilon r\right )|D^3_xu_n|^2,
\label{stimasug6}
\end{align}
\begin{align}
g_{7,n}\le
&a^3(t-s)T^2\eta^2r\frac{N}{2\varepsilon}|\nabla_xu_n|^2
+a^3(t-s)T^2\eta^4K_3\nu\frac{N^2}{2\varepsilon}|D^2_xu_n|^2\nonumber\\
&+2a^3(t-s)^3\eta^6\varepsilon \left (K_3\nu + r\right
)|D^3_xu_n|^2. \label{stimasug7}
\end{align}

The function $g_{8,n}$ can be estimated as follows:
\begin{equation}
g_{8,n}\le a|\nabla_xu_n|^2+2a^2(t-s)\eta^2|D^2_xu_n|^2+3a^3(t-s)^2\eta^4|D^3_xu_n|^2.
\label{stimasug8}
\end{equation}

Finally, taking Hypotheses \ref{ipos-1}(iv) and \ref{ipos-1}(vi) into account, we can
estimate
\begin{align*}
g_{9,n}\le\,& 2c_0v_n+T(1+T+T^2)u_n^2+
a^{\frac{3}{2}}(t-s)\left (2T\varrho_0^{-1}+\sqrt{a}+3T^2\varrho_0^{-1}a\right )\varrho^2\eta^2|\nabla_xu_n|^2\\
&+a^{\frac{5}{2}}(t-s)^2\left (2\varrho_0^{-1}+3T\varrho^{-1}_0+a^{\frac{3}{2}}\right )\eta^4\varrho^2|D^2_xu_n|^2\\
&+a^{\frac{7}{2}}(t-s)^3\left (6\varrho_0^{-1}+a^{\frac{5}{2}}\right
)\eta^6\varrho^2 |D^3_xu_n|^2.
\end{align*}

From \eqref{stimasug1}, \eqref{stimasug3}-\eqref{stimasug7} we
obtain, for any $t\in [0,T]$,
\begin{align}
g_n\le &\bigg\{-\nu_0+a +\nu\left [-1+aT\left
(2C'+2\frac{C''}{\varepsilon}+K_1\frac{N^2}{2\varepsilon}\right )
\right ]\nonumber\\
&\;\;\;+a(t-s)\eta^2\bigg [2d+aT(1+aT)\frac{N}{2\varepsilon}r\\
&\qquad\qquad\qquad\quad\;\,+\sqrt{a}\left
(2T\varrho_0^{-1}+\sqrt{a}(1+3T^2\varrho_0^{-1}\sqrt{a})\right
)\varrho^2
\bigg ]\bigg\}|\nabla_xu_n|^2\nonumber\\
&+ a\bigg \{-\nu_0+2a+\nu\bigg [-1+ 2\varepsilon
(4C''+K_1)+a^2T^2K_3\frac{N^2}{2\varepsilon}\nonumber\\
&\qquad\qquad\qquad\qquad\quad\;\,+aT\left
(4C'+4\frac{C''}{\varepsilon}+K_1\frac{N^2}{\varepsilon} +2K_2\right
)
\bigg ]\nonumber\\
&\;\;\;\;\;\;\;\;\;\;+a(t-s)\bigg [4d+\left
(2\varepsilon+3a\frac{N}{2\varepsilon}T
\right )r\nonumber\\
&\qquad\qquad\quad\;\;\;\;\;\;\;\;\;\;+\sqrt{a}\left
(2\varrho_0^{-1}+3T\varrho_0^{-1}+a^{\frac{3}{2}}\right )\varrho^2\bigg ]\eta^2\bigg \}(t-s)\eta^2|D^2_xu_n|^2\nonumber\\
&+ a^2\bigg\{-\nu_0+3a+\nu\bigg [-1+4\varepsilon(4C''+K_1) \nonumber\\
&\;\;\;\;\;\;\;\;\;\;\; +aT\left
(6C'+6\frac{C''}{\varepsilon}+3K_1\frac{N^2}{2\varepsilon}+2K_2(3+\varepsilon)\right
)\bigg ]\nonumber\\
&\;\;\;\;\;\;\;\;\;\;\;+a(t-s)\left (2(3d+4\varepsilon r)
+\sqrt{a}(6\varrho_0^{-1}+a^{\frac{5}{2}})\rho^2\right )\eta^2\bigg\}(t-s)^2\eta^4|D^3_xu_n|^2\nonumber\\
&+ 2a^3\left (-1+3\varepsilon(4C''+ K_1)\right )(t-s)^3\nu\eta^6
|D^4_xu_n|^2\nonumber\\
&+\left\{2c_0+T(1+T+T^2)\right\}v_n. \label{stimaC0-C3}
\end{align}
It is now easy to check that $\varepsilon$ and $a$ can be fixed sufficiently small such that all the terms in the
right-hand side of \eqref{stimaC0-C3}, but the last one, are
negative. We thus get
\begin{eqnarray*}
g_n(t,x)\le \big (2c_0+T(1+T+T^2)\big )v_n(t,x):=c_1v_n(t,x),
\end{eqnarray*}
for any $t\in [0,T]$ and any $x\in B(0,n)$. The maximum principle
now yields
\begin{equation}
|v_n(t,x)|\le e^{c_1(t-s)}\|\eta f\|_{\infty}^2\le
e^{c_1T}\|f\|_{\infty}^2,\qquad\;\,t\in (s,T],\; x\in B(0,n).
\label{conv-R-deriv}
\end{equation}
The proof of Theorem \ref{thm-3.2} shows that the function $u_n$
converges to $u$ in $C^{1,2}(D)$ for any compact set $D\subset
(s,T]\times\R^N$. We claim that $u$ is thrice continuously
differentiable with respect to the spatial variables in
$(s,T]\times\R^N$ and $D^3_xu_n$ converges to $D^3_xu$ as $n\to
+\infty$, locally uniformly in $(s,T]\times\R^N$. Indeed, since the
coefficients of the operator ${\mathscr A}$ are smooth, the interior
Schauder estimates imply that the first-order spatial derivatives of
the function $u_n$ are bounded in $C^{1+\delta/2,2+\delta}(D)$ for
any $D$ as above. Ascoli-Arzel\`a theorem now yields the claim.
Hence, taking the limit as $n\to+\infty$ in \eqref{conv-R-deriv},
estimate \eqref{stimasem} follows at once.

To prove \eqref{stimasem} in the other situations when $\alpha,\beta\in\N$ and $\alpha\le\beta$, it suffices to
apply the same arguments as above to the function
\begin{eqnarray*}
w(t,x)=\sum_{j=0}^{\beta}a^j(t-s)^{(j-\alpha)^{+}}\eta(x)^{2j}|D^ju_n(t,x)|^2,\quad\;\,t\in
(s,T],\;\, x\in B(0,n),
\end{eqnarray*}
where $(\cdot)^+$ denotes the positive part of the number in brackets.
\endproof

\subsubsection{\bf The case when $(\alpha,\beta)\notin\N\times\N$}

As it has been already claimed, to prove \eqref{stimasem} in the
general case we use an interpolation argument. It is well known
that, given four Banach spaces $X_1,X_2,Y_1,Y_2$, with $Y_i$
continuously embedded into $X_i$ ($i=1,2$), any linear operator $S$,
which is bounded from $X_1$ into $X_2$ and from $Y_1$ into $Y_2$, is
bounded from the interpolation space $(X_1,Y_1)_{\theta,\infty}$
into the interpolation space $(X_2,Y_2)_{\theta,\infty}$ for any
$\theta\in (0,1)$ and
\begin{equation}
\|S\|_{L((X_1,Y_1)_{\theta,\infty};(X_2,Y_2)_{\theta,\infty})}\le
\|S\|_{L(X_1,X_2)}^{1-\theta}\|S\|_{L(Y_1;Y_2)}^{\theta},
\label{stima-interp-gen}
\end{equation}
see e.g., \cite[pag. 25]{triebel}. We apply estimate
\eqref{stima-interp-gen} with $X_1=C_b(\R^N)$,
$X_2,Y_1,Y_2=C_b^3(\R^N)$ and $S=G(t,s)$. Since
$(C_b(\R^N),C^3_b(\R^N))_{\theta,\infty}=C^{3\theta}_b(\R^N)$ (for
any $\theta\in (0,1)$ such that $3\theta\notin\N$) and
$(C_b^3(\R^N),C^3_b(\R^N))_{\theta,\infty}=C^3_b(\R^N)$, with
equivalence of the corresponding norms (see e.g., \cite[Chapter 2,
Section 7, Theorem 1]{triebel}), from the results in Subsection
\ref{subsect-3.1}, we obtain \eqref{stimasem} with $\alpha\in (0,3)$
and $\beta=3$. A similar argument allows us to prove
\eqref{stimasem} also when $\alpha<\beta=1,2$ and $\alpha\notin\N$.

Now, we observe that the maximum principle yields
\begin{eqnarray*}
\|G(t,s)\|_{L(C_b(\R^N))}\le C,\qquad\;\, s\le t\le T,
\end{eqnarray*}
for some positive constant $C$. Hence, applying
\eqref{stima-interp-gen} with $X_1=X_2=C_b(\R^N)$ and
$Y_1=Y_2=C^3_b(\R^N)$, \eqref{stimasem} follows for any
$0\le\alpha=\beta\le 3$ such that $\alpha,\beta\notin\N$.

To prove \eqref{stimasem} in the general case, it now suffices to
fix $\alpha$ and $\beta\in [0,3]$, with $\alpha<\beta$,
$\alpha,\beta\notin\N$, and apply \eqref{stima-interp-gen} with
$X_1=X_2=Y_1=C^{\alpha}_b(\R^N)$, $Y_2=C^3_b(\R^N)$ and
$\theta=(3-\alpha)^{-1}(\beta-\alpha)$.

\begin{remark}
Let $I$ be a right-halfline and assume that Hypotheses \ref{ipos-1},
but \ref{ipos-1}(i), are satisfied with $[0,T]$ being replaced by
$I$. Further assume that
\begin{enumerate}
\item[{\rm (i')}]
the coefficients $q_{ij},b_j,c$ $(i,j=1,\ldots,N$) are thrice
continuously differentiable with respect to the spatial variables in
$I\times\R^N$ and their first-, second- and third-order spatial
derivatives are in $C^{\delta/2,\delta}(D)$ for some $\delta\in
(0,1)$ and any compact set $D\subset I\times\R^N$.
\end{enumerate}
Then, for any $\alpha,\beta\in [0,3]$, with $\alpha\le\beta$, there
exists a positive constant $C=C(\alpha,\beta)$ such that
\begin{equation}
\|G(t,s)f\|_{C^{\beta}_b(\R^N)}\le
C(t-s)^{-\frac{\beta-\alpha}{2}}e^{c_0(t-s)}\|f\|_{C^{\alpha}_b(\R^N)},
\qquad\;\,f\in C_b^{\alpha}(\R^N), \label{stimasem-bis}
\end{equation}
for any $s,t\in I$ with $s< t$. Here, $c_0$ is the constant in
Hypothesis \ref{ipos-1}(vii).

The proof of Theorem \ref{C0-C3} shows that
\begin{equation}
\|G(t,s)f\|_{C^{\beta}_b(\R^N)}\le
C_1(t-s)^{-\frac{\beta-\alpha}{2}}\|f\|_{C^{\alpha}_b(\R^N)},
\qquad\;\,f\in C_b^{\alpha}(\R^N), \label{stimasem-ter}
\end{equation}
for any $t,s\in I$ such that $0<t-s\le 1$ and some positive constant
$C_1$, independent of $s$ and $t$. On the other hand, if $t>s+1$, by
virtue of Proposition \ref{maximum} we can split
$G(t,s)f=G(t,t-1)G(t-1,s)f$. Hence, from \eqref{estim-u-om} and
\eqref{stimasem} with $\alpha=0$, $\beta=3$ we get
\begin{align}
\|G(t,s)f\|_{C^{\beta}_b(\R^N)}&=\|G(t,t-1)G(t-1,s)f\|_{C^{\beta}_b(\R^N)}\notag\\
& \le C_1\|G(t-1,s)f\|_{\infty}\notag\\
& \le C_1e^{c_0(t-s-1)}\|f\|_{\infty}\notag\\
& \le C_1e^{c_0(t-s-1)}\|f\|_{C^{\alpha}_b(\R^N)}.
\label{stimasem-quater}
\end{align}
Estimate \eqref{stimasem-bis} now follows immediately from
\eqref{stimasem-ter} and \eqref{stimasem-quater}.
\end{remark}

\subsection{Optimal Schauder estimates}
\label{subsect-2.3} Using the uniform estimates in Theorem
\ref{C0-C3}, we will prove an existence and uniqueness result for
the Cauchy problem
\begin{equation}
\left\{
\begin{array}{lll}
D_tu(t,x)=(\mathcal{A}u)(t,x)+g(t,x),\quad &t\in [0,T], &x\in\R^N,\\[1.5mm]
u(0,x)=f(x),&&x\in\R^N.
\end{array}
\right.
\label{nonom}
\end{equation}
as well as optimal Schauder estimates for its solution. For this
purpose, we introduce the following definition.
\begin{definition}
\label{defin-0} For any $\alpha\notin\N$, we denote by
$C^{0,\alpha}([0,T]\times\R^N)$ the set of all continuous functions
$f:[0,T]\times\R^N\to\R$ such that $f(t,\cdot)\in
C^{\alpha}_b(\R^N)$ for any $t\in [0,T]$ and
\begin{eqnarray*}
\|f\|_{C^{0,\alpha}([0,T]\times\R^N)}:=\sup_{t\in
[0,T]}\|f(t,\cdot)\|_{C^{\alpha}_b(\R^N)}<+\infty.
\end{eqnarray*}
\end{definition}

We state the main result of this first part of the paper.
\begin{theorem}
\label{thm-schauder} Suppose that Hypotheses $\ref{ipos-1}$ are
satisfied. Fix $\theta\in (0,1)$, $g\in
C^{0,\theta}([0,T]\times\R^N)$ and $f\in C^{2+\theta}_b(\R^N)$.
Then, problem \eqref{nonom} admits a unique bounded classical
solution. Moreover, $u(t,\cdot)\in C^{2+\theta}_b(\R^N)$ for any
$t\in [0,T]$ and there exists a positive constant $C$ such that
\begin{equation}
\|u\|_{C^{0,2+\theta}([0,T]\times\R^N)}\le C\left (
\|f\|_{C^{2+\theta}_b(\R^N)}+\|g\|_{C^{0,\theta}([0,T]\times\R^N)}\right
). \label{stima-teo-schauder}
\end{equation}
\end{theorem}

\begin{proof}
The proof can be obtained repeating almost verbatim the arguments in
the proof of \cite[Theorem 2]{lunardi-studia} (see also
\cite[Chapter 5]{libro}). For the reader's convenience we sketch
it.

The uniqueness part of the assertion is an immediate consequence of
the maximum principle in Proposition \ref{maximum}.

As far the existence part and the optimal Schauder estimates are
concerned, we show that the solution to problem \eqref{nonom} is
given by the variation-of-constants formula
\begin{equation}
u(t,x)=(G(t,0)f)(x)+\int_0^t(G(t,r)g(r,\cdot))(x)dr,\qquad\;\,t\in
[0,T],\;x\in\R^N, \label{variat}
\end{equation}
as in the classical case of bounded coefficients. Of course, it is
enough to consider the convolution term in \eqref{variat} (which we
denote by $v$). The main step of the proof consists in showing that
$v$ belongs to $C^{0,2+\theta}([0,T]\times\R^N)$ and
\begin{equation}
\|v\|_{C^{0,2+\theta}([0,T]\times\R^N)}\le \tilde
C\|g\|_{C^{0,\theta}([0,T]\times\R^N)}, \label{stima-v}
\end{equation}
for some positive constant $\tilde C$, independent of $g$. Estimate
\eqref{stima-v} follows from the interpolation argument in
\cite{lunardi-sem}, based on the uniform estimates of Subsection
\ref{subsect-2.2}. For any $\xi\in (0,1)$, $v(t,\cdot)$ is split
into the sum $v(t,\cdot)=a_{\xi}(t,\cdot)+b_{\xi}(t,\cdot)$ where
\begin{eqnarray*}
a_{\xi}(t,x)=\left\{
\begin{array}{ll}
\displaystyle\int_{t-\xi}^t(G(t,r)g(r,\cdot))(x)dr,\quad &\xi\in [0,t),\\[3.5mm]
\displaystyle\int_{0}^t(G(t,r)g(r,\cdot))(x)dr,\quad &{\rm
otherwise},
\end{array}
\right.
\end{eqnarray*}
and
\begin{eqnarray*}
b_{\xi}(t,x)=\left\{
\begin{array}{ll}
\displaystyle\int_0^{t-\xi}(G(t,r)g(r,\cdot))(x)dr,\quad &\xi\in [0,t),\\[3.5mm]
0,\quad &{\rm otherwise}.
\end{array}
\right.
\end{eqnarray*}
The uniform estimates in Theorem \ref{C0-C3} can be used to check
that $a_{\xi}(t,\cdot)$ and $b_{\xi}(t,\cdot)$ belong respectively
to $C^{\alpha}_b(\R^N)$ and $C^{2+\alpha}_b(\R^N)$ for any
$\alpha\in (\theta,1)$. Moreover,
\begin{eqnarray*}
\|a_{\xi}(t,\cdot)\|_{C^{\alpha}_b(\R^N)}+
\xi\|b_{\xi}(t,\cdot)\|_{C^{2+\alpha}_b(\R^N)}\le
C_1\xi^{1-(\alpha-\theta)/2}\|g\|_{C^{0,\theta}([0,T]\times\R^N)},
\end{eqnarray*}
for some positive constant $C_1$, independent of $\xi$ and $g$. This
estimate shows that $v(t,\cdot)$ belongs to the interpolation space
$(C^{\alpha}_b(\R^N),C^{2+\alpha}_b(\R^N))_{1-(\alpha-\theta)/2,\infty}$
for any $t\in [0,T]$, and
\begin{eqnarray*}
\|v(t,\cdot)\|_{(C^{\alpha}_b(\R^N),C^{2+\alpha}_b(\R^N))_{1-(\alpha-\theta)/2,\infty}}
\le C_1\|g\|_{C^{0,\theta}([0,T]\times\R^N)},
\end{eqnarray*}
for any $t\in [0,T]$.
Since $(C^{\alpha}_b(\R^N),C^{2+\alpha}_b(\R^N))_{1-(\alpha-\theta)/2,\infty}=C^{2+\theta}_b(\R^N)$
with equivalence of the corresponding norms (see \cite[Chapter 2, Section 7, Theorem 1]{triebel}), estimate \eqref{stima-v} follows.

Estimate \eqref{stima-v} combined with the fact that $v$ is
continuous in $[0,T]\times\R^N$, shows that the spatial derivatives
of $u$, up to the second-order, are continuous in $[0,T]\times\R^N$.

To conclude the proof, one just needs to show that the function $v$
is differentiable with respect to time in $[0,T]\times\R^N$ and
$D_tv={\mathscr A}v+g$, but this is rather straightforward and
the proof is omitted.
\end{proof}

\begin{remark}
\label{rem-schauder}
\begin{enumerate}[\rm (i)]
\item
The proof of Theorem \ref{thm-schauder} shows that the constant $C$
in \eqref{stima-teo-schauder} only depends on the constants in
\eqref{stimasem} which, in their turn, only depend on the constants
$C_j$, $K_j$, $L_j$ ($j=1,2,3$) as well as the ellipticity constant
$\nu_0$ in Hypotheses \ref{ipos-1}.
\item
Formula \eqref{variat}, Theorem \ref{thm-3.2} and the proof of
Theorem \ref{thm-schauder} show that the assumption $f\in C_b(\R^d)$
is enough for problem \eqref{nonom} to have a bounded classical
solution.
\end{enumerate}
\end{remark}

\section{The case when the diffusion coefficients are only
measurable in the pair $(t,x)$} \label{sect-3}
\setcounter{equation}{0} In this section, we consider some situation
in which the diffusion coefficients are bounded but not continuous
in $[0,T]\times\R^N$.

To state our standing assumptions and, then, the main result of this section,
let us give the following
definition, which is the counterpart of Definition \ref{defin-0} in
this new setting.

\begin{definition}
Fix $\alpha>0$.
\begin{enumerate}[\rm (i)]
\item
$M^{0,\alpha}([0,T]\times\R^N)$ denotes the space of all measurable
functions $f:[0,T]\times\R^N\to\R$ such that $f(t,\cdot)\in
C^{\alpha}(B(0,R))$ for any $t\in [0,T]$ and any $R>0$, and the
supremum of the $C^{\alpha}(B(0,R))$-norms of $f(t,\cdot)$, when $t$
runs in $[0,T]$, is finite for any $R>0$. Note that it may blow up
as $R\to +\infty$.
\item
$B^{0,\alpha}([0,T]\times\R^N)$ denotes the subset of
$M^{0,\alpha}([0,T]\times\R^N)$ of all bounded functions
$f:[0,T]\times\R^N\to\R$ such that $f(t,\cdot)\in
C^{\alpha}_b(\R^N)$ for any $t\in [0,T]$ and the supremum of the
$C^{\alpha}_b$-norms of $f(t,\cdot)$, when $t$ runs in $[0,T]$, is
finite. We norm $B^{0,\alpha}([0,T]\times\R^N)$ by setting
\begin{eqnarray*}
\|f\|_{B^{0,\alpha}([0,T]\times\R^N)}:=\sup_{t\in
[0,T]}\|f(t,\cdot)\|_{C^{\alpha}_b(\R^N)}.
\end{eqnarray*}
\end{enumerate}
\end{definition}

\begin{ipos}
\label{ipos-2} ~
\par
\noindent
\begin{enumerate}[\rm (i)]
\item
the coefficients $q_{ij}=q_{ji}$, $b_j$ $(i,j=1,\ldots,N)$ and $c$ belong
to $M^{0,3+\delta}([0,T]\times\R^N)$ for some $\delta\in (0,1)$.
\item
Hypotheses $\ref{ipos-1}(ii)$ to $\ref{ipos-1}(vii)$
are satisfied with $[0,T]$ being replaced by
${\mathcal D}$, where $[0,T]\setminus {\mathcal D}$ is a negligible
set. Moreover, for any $x\in\R^N$, the functions
$d(\cdot,x)$, $r(\cdot,x)$ and $\varrho(\cdot,x)$ are bounded and measurable in
$(0,T)$.
\end{enumerate}
\end{ipos}

Since the coefficients of the operator ${\mathscr A}$ are not
continuous, we do not expect that the Cauchy problem
\begin{equation}
\left\{
\begin{array}{lll}
D_tu(t,x)=(\mathcal{A}u)(t,x)+g(t,x),\quad &t\in [0,T], &x\in\R^N,\\[1.5mm]
u(0,x)=f(x),&&x\in\R^N.
\end{array}
\right. \label{nonom-bdd}
\end{equation}
has a solution $u$ with the smoothness properties in Theorem
\ref{thm-schauder} even if the data $f$ and $g$ are smooth. In the
spirit of \cite{KP,Lo1,Lo2}, we give the following definition of
solution to problem \eqref{nonom-bdd}.

\begin{definition}
\label{defin-1} Let $f\in C^2_b(\R^N)$ and $g$ be a bounded and
measurable function such that $g(t,\cdot)$ is continuous in $\R^N$
for any $t\in [0,T]$. A bounded function $u:[0,T]\times\R^N\to\R$ is called
a solution to \eqref{nonom-bdd} if the following conditions are
satisfied:
\begin{enumerate}[\rm (i)]
\item
the function $u$ is Lipschitz continuous in
$[0,T]\times\overline{B(0,R)}$ for any $R>0$, its first- and
second-order space derivatives are bounded and continuous functions
in $[0,T]\times\R^N$;
\item
$u(0,x)=f(x)$ for any $x\in\R^N$;
\item
there exists a set $G\subset [0,T]\times\R^N$, with negligible
complement, such that $D_tu(t,x)={\mathscr A}u(t,x)+g(t,x)$ for any
$(t,x)\in G$. Moreover, for any $x\in\R^N$, the set $G_x=\{t\in
[0,T]: (t,x)\in G\}$ is measurable with measure $T$.
\end{enumerate}
\end{definition}

Let us now prove the following lemmas which play a fundamental role
in the proof of the main result of this section.

\begin{lemma}
\label{lemma-1}
Let $f\in M^{0,\theta}([0,T]\times\R^N)$ for some $\theta\in (0,1)$. Then, the following properties hold.
\begin{enumerate}[\rm (i)]
\item
The function $f(\cdot,x)$ is measurable for any $x\in\R^N$.
\item
There exists a measurable set $C\subset [0,T]$, whose complement is
negligible in $[0,T]$, such that the function
$F:[0,T]\times\R^N\to\R$, defined by
\begin{eqnarray*}
F(t,x)=\int_0^tf(s,x)ds,\qquad\;\,(t,x)\in [0,T]\times\R^N,
\end{eqnarray*}
is differentiable with respect to $t$ in $C\times\R^N$ and, therein,
$D_tF=f$.
\end{enumerate}
\end{lemma}

\begin{proof}
(i). Since $f$ is measurable, the function $f(\cdot,x)$ is
measurable in $(0,T)$ for almost any $x\in\R^N$. Let us denote by
$H$ the set of all $x\in\R^N$ such that the function $f(\cdot,x)$ is
measurable, and prove that $H=\R^N$. For this purpose, we observe
that, since it has a negligible complement, $H$ is a dense subset of
$\R^N$. Hence, for any $x\in\R^N$, there exists a sequence
$(x_n)\subset H$ converging to $x$ in $\R^N$. Since $f\in
M^{0,\theta}([0,T]\times\R^N)$, the function $f(t,\cdot)$ is
H\"older continuous in $B(0,R)$ of exponent $\theta$, uniformly with
respect to $t\in [0,T]$, where $R:=\max_{n\in\N}|x_n|+1$. Hence,
there exists a positive constant $C$, independent of $n$ and $t$,
such that
\begin{eqnarray*}
|f(t,x)-f(t,x_n)|\le C|x-x_n|^{\theta},\qquad t\in [0,T].
\end{eqnarray*}
This shows that $f(\cdot,x_n)$ converges to the function
$f(\cdot,x)$ uniformly in $[0,T]$. Consequently, $f(\cdot,x)$ is
measurable in $[0,T]$.

(ii). Since, by step (i), $f(\cdot,x)\in L^{\infty}(0,T)$
for any fixed $x\in\R^N$, the function $F$ is well defined and, for any $x\in\R^N$, there exists
a measurable set $C_x$, with negligible complement, such that $F(\cdot,x)$ is differentiable in
$C_x$ and $D_tF(t,x)=f(t,x)$ for any $t\in C_x$.

Let us set $C=\cap_{x\in\mathbb Q^N}C_x$ and fix $(t,x)\in
C\times\R^N$. Further, let $(x_n)\subset\mathbb Q^N$ converge to $x$
as $n\to +\infty$. Then, for any $h\in\R\setminus\{0\}$, we can
estimate
\begin{align*}
&\bigg |\frac{F(t+h,x)-F(t,x)}{h}-f(t,x)\bigg |\\
\le &
\frac{1}{|h|}\left |\int_t^{t+h}|f(s,x)-f(t,x)|ds\right |\\
\le&\frac{1}{|h|}\left |\int_t^{t+h}|f(s,x)-f(s,x_n)|ds\right | +
\frac{1}{|h|}\left |\int_t^{t+h}|f(s,x_n)-f(t,x_n)|ds\right |\\
&+|f(t,x_n)-f(t,x)|\\
\le & 2\sup_{t\in
[0,T]}\|f(t,\cdot)\|_{C^{\theta}(B(0,R))}|x-x_n|^{\theta} +
\frac{1}{|h|}\left |\int_t^{t+h}|f(s,x_n)-f(t,x_n)|ds\right |,
\end{align*}
where $R:=\max_{n\in\N}|x_n|+1$. Letting $h\to 0$, we get
\begin{eqnarray*}
\limsup_{h\to 0}\left |\frac{F(t+h)-F(t)}{h}-f(t,x)\right |\le
2\|f\|_{B^{0,\theta}([0,T]\times\R^N)}|x-x_n|^{\theta},
\end{eqnarray*}
for any $n\in\N$, which implies that $F$ is differentiable with
respect to time at $(t,x)$ and $D_tF(t,x)=f(t,x)$. This completes
the proof.
\end{proof}

\begin{lemma}
\label{lemma-2} Assume that Hypotheses $\ref{ipos-2}$ are satisfied.
Then, there exist sequences $(q_{ij}^{(n)})$, $(b_{j}^{(n)})$ and
$(c^{(n)})$ with the following properties:
\begin{enumerate}[\rm (i)]
\item
$q_{ij}^{(n)}$, $b_{j}^{(n)}$ $(i,j=1,\ldots,N)$, $c^{(n)}$ and
their spatial derivatives, up to the third-order, belong to
$C^{\delta/2,\delta}([0,T]\times B(0,R))$ for any $R>0$;
\item
there exists a measurable set ${\mathcal E}\subset [0,T]$, whose
complement in $[0,T]$ is negligible, such that, for any
$i,j=1,\ldots,N$, $q^{(n)}_{ij}$, $b_{ij}^{(n)}$ and $c^{(n)}$
converge pointwise in ${\mathcal E}\times\R^N$, respectively to
$q_{ij}$, $b_{j}$ and $c$, as $n\to +\infty$;
\item
for any $n\in\N$, the functions $q_{ij}^{(n)}$, $b_j^{(n)}$
$(i,j=1,\ldots,N)$ and $c^{(n)}$ satisfy Hypotheses
$\ref{ipos-1}(ii)$ to $\ref{ipos-1}(vii)$ with the functions
$\nu,d,r,\varrho$ being replaced by new functions
$\nu_n,d_n,r_n,\varrho_n$ and the same constants $C,L_1,L_2,L_3$.
Moreover, there exist two positive constants $\hat\nu_0$ and $\hat
c_0$ such that $\nu_n(t,x)\ge\hat\nu_0$ and $c_n(t,x)\le \hat c_0$
for any $(t,x)\in [0,T]\times\R^N$ and any $n\in\N$.
\end{enumerate}
\end{lemma}

\begin{proof}
By Lemma \ref{lemma-1}, the functions $q_{ij}(\cdot,x)$, $b_j(\cdot,x)$ $(i,j=1,\ldots,N)$ and
$c(\cdot,x)$ are in $L^{\infty}(0,T)$ for any $x\in\R^N$.
Thus, for any $n\in\N$, we can define the functions $q^{(n)}_{ij}$, $b^{(n)}_{ij}$ and $c^{(n)}$ by setting
\begin{align*}
q_{ij}^{(n)}(t,x)&=\left (\frac{n}{4\pi}\right )^{\frac{1}{2}}
\int_0^Tq_{ij}(\tau,x)\exp\left (-\frac{n}{4}|t-\tau|^2\right
)d\tau,\\
b_j^{(n)}(t,x)&=\left (\frac{n}{4\pi}\right )^{\frac{1}{2}}
\int_0^Tb_j(\tau,x)\exp\left (-\frac{n}{4}|t-\tau|^2\right )d\tau,\\
c^{(n)}(t,x)&=\left (\frac{n}{4\pi}\right )^{\frac{1}{2}}
\int_0^Tc(\tau,x)\exp\left (-\frac{n}{4}|t-\tau|^2\right )d\tau,
\end{align*}
for any $(t,x)\in [0,T]\times\R^N$ and any $i,j=1,\ldots,N$.
Clearly, $q^{(n)}_{ij}$, $b^{(n)}_j$ and $c^{(n)}$ and their spatial
derivatives, up to the third-order, belong to
$C^{\delta/2,\delta}([0,T]\times B(0,R))$ for any $i,j=1,\ldots,N$ and
any $R>0$.

Let us prove that, for any $i,j=1,\ldots,N$, $q_{ij}^{(n)}$
converges pointwise in ${\mathcal E}\times\R^N$ to $q_{ij}$, for
some measurable set ${\mathcal E}\subset [0,T]$ whose complement is
negligible. Then, the same argument can be applied to prove the
convergence of $b_j^{(n)}$ and $c^{(n)}$ to $b_j$ and $c$,
respectively.

Since $q_{ij}^{(n)}(\cdot,x)\to q_{ij}(\cdot,x)$ in $L^p(0,T)$ for
any $p\in [1,+\infty)$, any $x\in\R^N$ and any $i,j=1,\ldots,N$, we
can find out an increasing sequence $(n_k^x)\subset\N$ such that the
subsequence $q_{ij}^{(n_k^x)}(t,x)$ converges to $q_{ij}(t,x)$ as
$n$ tends to $+\infty$ almost everywhere in $(0,T)$. By a classical
diagonal procedure, we can determine an increasing sequence
$(n_k)\subset\N$ and a measurable set ${\mathcal E}\subset [0,T]$,
whose complement is negligible in $[0,T]$, such that
\begin{eqnarray}
\lim_{k\to
+\infty}q_{ij}^{(n_k)}(t,x)=q_{ij}(t,x),\qquad\;\,(t,x)\in {\mathcal
E}\times{\mathbb Q}^N,\;\,i,j=1,\ldots,N. \label{lim-raz}
\end{eqnarray}

Let us now show that we can extend \eqref{lim-raz} to any
$(t,x)\in{\mathcal E}\times\R^N$. For this purpose, we fix
$(t,x)\in{\mathcal E}\times\R^N$ and a sequence $(x_m)\subset
{\mathbb Q}^N$ converging to $x$ as $m$ tends to $+\infty$. Since
\begin{eqnarray*}
\sup_{t\in (0,T)}\|q_{ij}^{(n_k)}(t,\cdot)\|_{C^{\theta}(B(0,R))}\le
\sup_{t\in (0,T)}\|q_{ij}(t,\cdot)\|_{C^{\theta}(B(0,R))},
\end{eqnarray*}
for any $i,j=1,\ldots,N$, where $R=1+\sup_{m\in\N}|x_m|$, we can
write
\begin{align}
&|q^{(n_k)}_{ij}(t,x)-q_{ij}(t,x)|\nonumber\\
\le &\,|q^{(n_k)}_{ij}(t,x_m)-q^{(n_k)}_{ij}(t,x)|
+|q^{(n_k)}_{ij}(t,x_m)-q_{ij}(t,x_m)|+|q_{ij}(t,x_m)-q_{ij}(t,x)|\nonumber\\
\le\,& 2\sup_{t\in
(0,T)}\|q_{ij}(t,\cdot)\|_{C^{\theta}(B(0,R))}|x-x_m|^{\theta}+
|q^{(n_k)}_{ij}(t,x_m)-q_{ij}(t,x_m)|, \label{stima-q}
\end{align}
for any $k,m\in\N$. Taking, first, the limsup as $k\to +\infty$ in
the first- and last-side of \eqref{stima-q}, and then letting $m\to
+\infty$, \eqref{lim-raz} follows, for any $x\in\R^N$. Property (ii)
is proved.

Let us now prove property (iii). Taking Hypotheses \ref{ipos-1}(i)
into account, we get
\begin{eqnarray*}
\langle Q^{(n)}(t,x)\xi,\xi\rangle &\hs{5}\ge\hs{5}& |\xi|^2\left
(\frac{n}{4\pi}\right )^{\frac{1}{2}}
\int_0^T\nu(\tau,x)e^{-\frac{n}{4}|t-\tau|^2}d\tau:=\nu_n(t,x)|\xi|^2,
\end{eqnarray*}
for any $\xi\in\R^N$, any $(t,x)\in [0,T]\times\R^N$ and any
$n\in\N$, where $Q^{(n)}=(q_{ij}^{(n)})$. Note that the function
$\nu_n$ can be bounded from below in $[0,T]\times\R^N$ by a positive
constant, independent of $n$. Indeed,
\begin{eqnarray}
\nu_n(t,x) &\hs{5}\ge\hs{5}& \nu_0\left (\frac{n}{4\pi}\right
)^{\frac{1}{2}}\bigg\{\int_0^te^{-\frac{n}{4}\tau^2}d\tau
+\int_{0}^{T-t}e^{-\frac{n}{4}\tau^2}d\tau\bigg\}\nonumber\\
&\hs{5}\ge\hs{5}&\nu_0
\left (\frac{n}{4\pi}\right )^{\frac{1}{2}}\int_0^{\frac{T}{2}}e^{-\frac{n}{4}s^2}ds\nonumber\\
&\hs{5}\ge\hs{5}&
\frac{\nu_0}{2\sqrt{\pi}}\int_0^{\frac{T}{2}}e^{-\frac{1}{4}s^2}ds,
\label{stima-unif-nu-n}
\end{eqnarray}
for any $(t,x)\in [0,T]\times\R^N$.

Now, an easy computation shows that, for any $n\in\N$, the functions
$q_{ij}^{(n)}$, $b_j^{(n)}$ $(i,j=1,\ldots,N$) and $c^{(n)}$ satisfy
Hypothesis \ref{ipos-1}(iii) and \ref{ipos-1}(v) as well as
conditions \eqref{dissip}-\eqref{cond-deriv-c}, with the same
constants $C_j$ and $K_j$ ($j=1,2,3$) and $\nu,d,r,\varrho$ being
replaced with the functions $\nu_n,d_n,r_n,\varrho_n$, where
\begin{align*}
d_n(t,x)&=\left (\frac{n}{4\pi}\right )^{\frac{1}{2}}
\int_0^Td(\tau,x)\exp\left (-\frac{n}{4}|t-\tau|^2\right )d\tau,\\
r_n(t,x)&=\left (\frac{n}{4\pi}\right )^{\frac{1}{2}}
\int_0^Tr(\tau,x)\exp\left (-\frac{n}{4}|t-\tau|^2\right )d\tau,\\
\varrho_n(t,x)&=\left (\frac{n}{4\pi}\right )^{\frac{1}{2}}
\int_0^T\varrho(\tau,x)\exp\left (-\frac{n}{4}|t-\tau|^2\right )d\tau,
\end{align*}
for any $(t,x)\in [0,T]\times\R^N$. Arguing as in the proof of
\eqref{stima-unif-nu-n}, we can easily show that
\begin{eqnarray*}
\varrho_n(t,x)\ge
\frac{\varrho_0}{2\sqrt{\pi}}\int_0^{\frac{T}{2}}e^{-\frac{1}{4}s^2}ds.
\end{eqnarray*}
Moreover, integrating condition \eqref{cond-funz} we get
\begin{eqnarray*}
d_n(t,x)+L_1r_n(t,x)+L_2\left (\frac{n}{4\pi}\right )^{\frac{1}{2}}\int_0^T\varrho^2(\tau,x)\exp\left (-\frac{n}{4}|t-\tau|^2\right )d\tau\le L_3\nu_n(t,x),
\end{eqnarray*}
for any $(t,x)\in [0,T]\times\R^N$ and any $n\in\N$. H\"older inequality yields
\begin{align*}
\left (\frac{n}{4\pi}\right )^{\frac{1}{2}}\int_0^T\varrho^2(\tau,x)\exp\left (-\frac{n}{4}|t-\tau|^2\right )d\tau&\ge
\frac{n}{4\pi}\left (\int_0^T\varrho(\tau,x)\exp\left (-\frac{n}{4}|t-\tau|^2\right )d\tau\right )^2\\
&=(\varrho^{(n)}(t,x))^2,
\end{align*}
for any $(t,x)$ as above. Hence, condition \eqref{cond-funz} is satisfied with
$d,r,\varrho$ being replaced by $d_n,r_n,\varrho_n$ and the same constants $L_1,L_2,L_3$.

Next, we observe that $c^{(n)}$ satisfies Hypothesis
\ref{ipos-1}(iv) with $c_0$ being replaced by
\begin{eqnarray*}
\frac{c_0}{2\sqrt{\pi}}\int_0^{\frac{T}{2}}e^{-\frac{1}{4}s^2}ds.
\end{eqnarray*}

Finally, integrating condition \eqref{liapunov} with respect to time, we easily deduce that
\begin{eqnarray*}
\sup_{(t,x)\in [0,T]\times\R^N}\!\{({\mathscr
A}^{(n)}\varphi)(t,x)\hs{1}-\hs{1}\lambda\varphi(x)\}\le
\sup_{(t,x)\in (0,T)\times\R^N}\!\{({\mathscr
A}\varphi)(t,x)\hs{1}-\hs{1}\lambda\varphi(x)\}<+\infty,
\end{eqnarray*}
for any $n\in\N$, where by ${\mathscr A}^{(n)}$ we have denoted the elliptic operator
whose coefficients are $q_{ij}^{(n)}$, $b_j^{(n)}$ ($i,j=1,\ldots,N$) and $c^{(n)}$. This completes the proof.
\end{proof}

We now consider the following maximum principle, which generalizes Proposition \ref{maximum}
to the case of noncontinuous coefficients.

\begin{proposition}\label{maxprinc}
Assume that Hypotheses $\ref{ipos-2}$ are satisfied and let $u$ be a
solution to problem \eqref{nonom-bdd}, in the sense of Definition
$\ref{defin-1}$, corresponding to $f\in C_b^2(\R^N)$ and a bounded
and measurable function $g:[0,T]\times\R^N\to\R$ such that the
function $g(t,\cdot)$ is continuous for any $t\in [0,T]$. If $f\le
0$ and $g\le 0$, then $u\le 0$.
\end{proposition}

\begin{proof}
To begin with, we observe that there exists a positive function
$\hat\varphi:\R^N\to\R$ which blows up as $|x|\to +\infty$ and
${\mathscr A}\hat\varphi-\hat\lambda\hat\varphi<0$ in $[0,T]\times\R^N$ for
some $\hat\lambda>0$.
It suffices to replace in Hypothesis \ref{ipos-1}(vii) the function
$\varphi$ and the positive constant $\lambda$, respectively, with
the function $\hat\varphi=\varphi+C$ and
$\hat\lambda=\max\{\lambda,2c_0\}$, and to take $C$ sufficiently
large.

Let $u$ be a solution to problem \eqref{nonom-bdd}. For any $n\in\N$, we introduce the function
$v_n:[0,T]\times\R^N\to\R$
defined by
\begin{eqnarray*}
v_n(t,x)=e^{-\hat\l t}u(t,x)-\frac{1}{n}\hat\varphi(t,x),\qquad\;\,(t,x)\in [0,T]\times\R^N.
\end{eqnarray*}
Clearly, $v_n\in W^{1,1}_{\infty}((0,T)\times B(0,R))\subset
W^{1,2}_{N+1}((0,T)\times B(0,R))$ for any $R>0$, it satisfies the
differential inequality $D_tv_n-{\mathscr A}v_n+\hat\lambda v_n\le
e^{-\hat\l \cdot}g$ in the sense of distributions, and
$v_n(0,\cdot)\le f$. Hence, the Nazarov-Ural'tseva maximum principle
(see \cite[Theorem 1]{nazarov-uraltseva}) may be applied. It yields
\begin{eqnarray*}
v_n(t,x)\le \sup_{(t,x)\in (0,T)\times\partial B(0,R)}v_n^+(t,x),
\end{eqnarray*}
for any $(t,x)\in [0,T]\times \overline{B(0,R)}$ and any $R>0$.
Here, $v_n^+$ denotes the positive part of the function $v_n$.
Since, for any $n\in\N$, $v_n(t,x)$ tends to $-\infty$ as $|x|\to
+\infty$, uniformly with respect to $t$, $v_n(t,x)\le 0$ for any
$(t,x)\in [0,T]\times\R^N$. Letting $n\to +\infty$, yields the
assertion.
\end{proof}

We are now in a position to prove the main result of this section.
\begin{theorem}
\label{thm-main-2} Let Hypotheses $\ref{ipos-2}$ be satisfied. Fix
$\theta\in (0,1)$ and suppose that $f\in C^{2+\theta}_b(\R^N)$ and
$g\in B^{0,\theta}([0,T]\times\R^N)$. Then, the Cauchy problem
\eqref{nonom-bdd} admits a unique solution $u$, in the sense of
Definition $\ref{defin-1}$. The function $u$ belongs to
$B^{0,2+\theta}([0,T]\times\R^N)$ and there exists a positive
constant $C$, independent of $f$ and $g$, such that
\begin{eqnarray*}
\|u\|_{B^{0,2+\theta}([0,T]\times\R^N)}\le C\left
(\|f\|_{C^{2+\theta}_b(\R^N)}+\|g\|_{B^{0,\theta}([0,T]\times\R^N)}\right
).
\end{eqnarray*}
\end{theorem}

\begin{proof}
The uniqueness of the solution is a straightforward consequence of
the maximum principle in Proposition \ref{maxprinc}.

To prove that problem \eqref{nonom-bdd} actually admits a solution in the sense
of Definition \ref{defin-1}, we use an approximation argument. For any $n\in\N$, we
introduce the operator ${\mathscr A}^{(n)}$ defined by
\begin{eqnarray*}
{\mathscr A}^{(n)}=\sum_{i,j=1}^Nq_{ij}^{(n)}D^2_{ij}+\sum_{j=1}^Nb_j^{(n)}D_j+c^{(n)}I,
\end{eqnarray*}
where the coefficients $q_{ij}^{(n)}$, $b_j^{(n)}$
($i,j=1,\ldots,N$) and $c^{(n)}$ are defined in Lemma \ref{lemma-2}.
We further approximate the function $g$ by a sequence of functions
$g^{(n)}\in C^{0,\theta}([0,T]\times\R^N)$, defined by
\begin{eqnarray*}
g^{(n)}(t,x)=\left (\frac{n}{4\pi}\right )^{\frac{1}{2}}
\int_0^Tg(\tau,x)\exp\left (-\frac{n}{4}|t-\tau|^2\right
)d\tau,\qquad\;\,(t,x)\in [0,T]\times\R^N.
\end{eqnarray*}
Clearly, $\|g^{(n)}\|_{C^{0,\theta}([0,T]\times\R^N)}\le
\|g\|_{B^{0,\theta}([0,T]\times\R^N)}$ for any $n\in\N$. Moreover,
by the proof of Lemma \ref{lemma-2}, there exists a set ${\mathcal
D}$, whose complement is negligible in $[0,T]$, such that
$g^{(n)}(t,x)$ tends to $g(t,x)$ as $n\to +\infty$, for any
$(t,x)\in {\mathcal D}\times\R^N$. Again, Lemma \ref{lemma-2}
implies that the coefficients of the operator ${\mathscr A}^{(n)}$
satisfy Hypotheses \ref{ipos-1}, with constants independent of $n$.
Hence, the Cauchy problem
\begin{equation}
\left\{
\begin{array}{lll}
D_tu(t,x)={\mathscr A}^{(n)}u(t,x)+g^{(n)}(t,x),\quad &t\in [0,T], &x\in\R^N,\\[1.5mm]
u(0,x)=f(x), &&x\in\R^N,
\end{array}
\right. \label{pb-approx}
\end{equation}
admits a unique classical solution $u_n$ which belongs to
$C^{0,2+\theta}([0,T]\times\R^N)$ and
\begin{align}
\|u_n\|_{C^{0,2+\theta}([0,T]\times\R^N)}&\le C_1\left
(\|f\|_{C^{2+\theta}_b(\R^N)}
+\|g^{(n)}\|_{C^{0,\theta}([0,T]\times\R^N)}\right )\nonumber\\
&\le C_1\left (\|f\|_{C^{2+\theta}_b(\R^N)}
+\|g\|_{B^{0,\theta}([0,T]\times\R^N)}\right ),
\label{schauder-estim-n}
\end{align}
for some positive constant $C_1$, independent of $n$ (see Remark
\ref{rem-schauder}(i)).

From the differential equation in \eqref{pb-approx} and the estimate
\eqref{schauder-estim-n} it follows that, for any $R>0$, the
sequence $(D_tu_n)$ is bounded in $[0,T]\times B(0,R)$. As a
byproduct, $\|u_n\|_{{\rm Lip}([0,T]\times B(0,R))}\le C$ for some
positive constant, independent of $n$. Using an interpolation
argument, we can now show that the functions $D_iu_n$ and
$D_{ij}u_n$ ($i,j=1,\ldots,N$, $n\in\N$) are equibounded and
equicontinuous in $[0,T]\times\overline{B(0,R)}$. Indeed, it is well
known that there exists a positive constant $K$ such that
\begin{align*}
\|\psi\|_{C^1(B(0,R))}&\le
K\|\psi\|_{C(B(0,R))}^{\frac{1+\theta}{2+\theta}}\|\psi\|_{C^{2+\theta}(B(0,R))}^{\frac{1}{2+\theta}},\\
\|\psi\|_{C^2(B(0,R))}&\le
K\|\psi\|_{C(B(0,R))}^{\frac{\theta}{2+\theta}}\|\psi\|_{C^{2+\theta}(B(0,R))}^{\frac{2}{2+\theta}},
\end{align*}
for any $\psi\in C^{2+\theta}(B(0,R))$. It follows that
\begin{align*}
&\|\nabla_xu_n(t,\cdot)-\nabla_xu_n(s,\cdot)\|_{C(B(0,R))}\\
\le&
K\|u_n(t,\cdot)-u_n(s,\cdot)\|_{C(B(0,R))}^{\frac{1+\theta}{2+\theta}}
\|u_n(t,\cdot)-u_n(s,\cdot)\|_{C^{2+\theta}(B(0,R))}^{\frac{1}{2+\theta}}\\
\le&
K_{\theta}\|u_n\|_{C^{0,2+\theta}([0,T]\times\R^N)}^{\frac{1}{2+\theta}}
[u_n]_{{\rm Lip}([0,T]\times B(0,R))}^{\frac{1+\theta}{2+\theta}}|t-s|^{\frac{1+\theta}{2+\theta}}\\
\le &K'_{\theta}\left (\|f\|_{C^{2+\theta}_b(\R^N)}+\|g\|_{B^{0,\theta}([0,T]\times\R^N)}\right
)|t-s|^{\frac{1+\theta}{2+\theta}},
\end{align*}
and, similarly,
\begin{align*}
&\|D^2u_n(t,\cdot)-D^2u_n(s,\cdot)\|_{C(B(0,R))}\\
\le&
K''_{\theta} \left (\|f\|_{C^{2+\theta}_b(\R^N)}+\|g\|_{B^{0,\theta}([0,T]\times\R^N)}\right )
|t-s|^{\frac{\theta}{2+\theta}},
\end{align*}
for any $s,t\in [0,T]$. Here, $K_{\theta}$, $K_{\theta}'$ and
$K_{\theta}''$ are positive constants which may blow up as $R\to
+\infty$. The previous estimates show that, for any
$i,j=1,\ldots,N$, the sequences $(D_ju_n)$ and $(D_{ij}u_n)$ are
equibounded and equicontinuous in $[0,T]\times B(0,R)$. Since $R$ is
arbitrary, by Ascoli-Arzel\`a theorem there exists a function $u\in
C^{0,2}([0,T]\times\R^N)$ and a subsequence $(u_{n_k})$ converging
to  $u$ in $C^{0,2}([0,T]\times K)$ for any compact set
$K\subset\R^N$. Moreover, $u$ belongs to ${\rm
Lip}([0,T]\times\overline{B(0,R)})$ for any $R>0$. Hence, for any
$x\in\R^N$, the function $u(\cdot,x)$ is differentiable almost
everywhere in $(0,T)$. Clearly, $u(0,\cdot)\equiv f$ since
$u_{n_k}(0,\cdot)\equiv f$ for any $k\in\N$.

To complete the proof, let us show that $u$ is differentiable with
respect to $t$ in $G\times\R^N$, for some measurable set $G\subset
[0,T]$, whose complement is negligible, and $D_tu(t,x)={\mathscr
A}u(t,x)+g(t,x)$ for such values of $t$. For this purpose, we
observe that, for any $(t,x)\in [0,T]\times\R^N$, it holds that
\begin{eqnarray}
u_{n_k}(t,x)= f(x)+\int_0^t\left ({\mathscr
A}^{(n_k)}u_{n_k}(s,x)+g^{(n_k)}(s,x)\right )ds.
\label{estim-der-t}
\end{eqnarray}
Taking Lemma \ref{lemma-2} into account, we can let $k\to +\infty$ in both the sides of
\eqref{estim-der-t}. This yields
\begin{eqnarray*}
u(t,x)=f(x)+\int_0^t\left ({\mathscr A}u(s,x)+g(s,x)\right
)ds,\qquad\;\,(t,x)\in [0,T]\times\R^N.
\end{eqnarray*}
The assumptions on the coefficients of the operator ${\mathscr A}$
and the regularity properties of the function $u$, already proved,
imply that the function ${\mathscr A}u+g$ satisfies the assumptions
of Lemma \ref{lemma-1}(ii). Therefore, there exists a set $G\subset
[0,T]$, whose complement is negligible in $[0,T]$, such that $u$ is
differentiable in $G\times\R^N$ with respect to the time variable
and $D_tu={\mathscr A}u+g$ in $G\times\R^N$. This accomplishes the
proof.
\end{proof}

Taking Remark \ref{rem-schauder}(i) into account and using the same
argument as in the proof of Theorem \ref{thm-main-2}, one can show
that Theorem \ref{thm-schauder} holds true also under a slightly
weaker regularity assumption on the coefficients of the operator
${\mathscr A}$. More precisely,

\begin{theorem}
\label{thm-schauder-cont-2} Suppose that Hypotheses $\ref{ipos-1}$
are satisfied, but $\ref{ipos-1}(i)$, in which the space
$C^{\delta/2,\delta}((0,T)\times B(0,R))$ is replaced with
$C^{0,\delta}([0,T]\times B(0,R))$ $($defined as in Definition
$\ref{defin-0}$, with $\R^N$ replaced by $B(0,R))$. Then, the
assertion of Theorem \ref{thm-schauder} holds true.
\end{theorem}

\section{An example}
\label{sect-4}
 In this section, we exhibit a class of nonautonomous
elliptic operators with unbounded coefficients that satisfy the
assumptions of Theorems \ref{thm-main-2} and
\ref{thm-schauder-cont-2}.

Let ${\mathscr A}$ be the elliptic operator defined by
\begin{align*}
({\mathscr A}\psi)(t,x)=
&(1+|x|^2)^p\sum_{i,j=1}^Nq_{ij}^{(0)}(t,x)D_{ij}\psi(x)+
b^{(0)}(t)(1+|x|^2)^{q} \langle x, \nabla_x \psi(x)
\rangle\\
&+(c^{(0)}(t,x)-|x|^{2r})\psi(x),
\end{align*}
for any $(t,x)\in [0,T]\times\R^N$, on smooth functions
$\psi:\R^N\to\R$.
We assume that the coefficients of the operator ${\mathscr A}$ satisfy either
\begin{ipos}
\label{ipos-3}
~
\par
\begin{enumerate}[\rm (i)]
\item
The functions $q_{ij}^{(0)}$ $(i,j=1,\ldots,N$) and $c^{(0)}$ are
thrice continuously differentiable with respect to the spatial
variables in $[0,T]\times\R^N$. Moreover, they are bounded together
with their spatial derivatives up to the third-order. Further, there
exists $\delta\in (0,1)$ such that, for any $R>0$, the third-order
spatial derivatives of $q_{ij}^{(0)}$ ($i,j=1,\ldots,N$) and
$c^{(0)}$ are $\delta$-H\"older continuous in $B(0,R)$, uniformly
with respect to $t\in [0,T]$;
\item
there exists a positive constant $\nu_0$ such that
\begin{eqnarray*}
\langle Q^{(0)}(t,x)\xi,\xi\rangle\ge\nu_0|\xi|^2,\qquad\;\,t\in
[0,T],\;\,x,\xi\in\R^N;
\end{eqnarray*}
\item
$p,q,r\in\N\cup\{0\}$ satisfy $p\le q$;
\item
the function $b^{(0)}$ is continuous and $b^{(0)}(t)<0$ for any
$t\in [0,T]$,
\end{enumerate}
\end{ipos}
\noindent
or
\begin{ipos}
\label{ipos-4}
~
\par
\begin{enumerate}[\rm (i)]
\item
The functions $q_{ij}^{(0)}$ $(i,j=1,\ldots,N$) and $c^{(0)}$ are
$B^{0,3}([0,T]\times\R^N)$ and the third-order spatial derivatives
are in $M^{0,\delta}([0,T]\times\R^N)$ for some $\delta\in (0,1)$;
\item
there exists a positive constant $\nu_0$ such that
\begin{eqnarray*}
\langle Q^{(0)}(t,x)\xi,\xi\rangle\ge\nu_0|\xi|^2,\qquad\;\,t\in
F,\;\,x,\xi\in\R^N,
\end{eqnarray*}
where $F$ is a measurable set whose complement is negligible in $[0,T]$;
\item
$p,q,r\in\N\cup\{0\}$ satisfy $p\le q$;
\item
the function $b^{(0)}$ is bounded and measurable in $(0,T)$ and
there exists a negative constant $b_0$ such that $b^{(0)}(t)\le b_0$
for almost any $t\in (0,T)$.
\end{enumerate}
\end{ipos}

Let us check that, under Hypotheses \ref{ipos-3}, the operator
${\mathscr A}$ satisfies the assumptions of Theorem
\ref{thm-schauder-cont-2}. The same arguments will show that, if
${\mathscr A}$ satisfies Hypotheses \ref{ipos-4}, then it satisfies
also Hypotheses \ref{ipos-2}, so that Theorem \ref{thm-main-2} holds
true.

It is immediate to check that the coefficients $q_{ij}$, $b_j$ and
$c$, where
\begin{align*}
&q_{ij}(t,x)=q_{ij}^{(0)}(t,x)(1+|x|^2)^p,\\
&b_j(t,x)=b^{(0)}(t)x_j(1+|x|^2)^q,\\
&c(t,x)=c^{(0)}(t,x)-|x|^{2r},
\end{align*}
for any $(t,x)\in [0,T]\times\R^N$ and any $i,j=1,\ldots,N$, are
thrice continuously differentiable with respect to $x$, and, for any
$R>0$, the  third-order derivatives are H\"older continuous of
exponent $\delta$ with respect to the variable $x\in B(0,R)$,
uniformly with respect to $t\in [0,T]$. Similarly, checking
Hypotheses \ref{ipos-1}(ii) to \ref{ipos-1}(v) is an easy task. As
far as Hypothesis \ref{ipos-1}(vi) is concerned, we observe that
\begin{eqnarray*}
\langle Db(t,x)\xi,\xi\rangle =b^{(0)}(t)(1+|x|^2)^{q-1}\left
((1+|x|^2)|\xi|^2+2q\langle \xi,x\rangle^2\right ),
\end{eqnarray*}
for any $t\in [0,T]$ and any $x,\xi\in\R^N$. Since $b$ is negative
in $[0,T]$, we can estimate
\begin{eqnarray*}
\langle Db(t,x)\xi,\xi\rangle\le
b_0(1+|x|^2)^q|\xi|^2:=d(t,x)|\xi|^2,
\end{eqnarray*}
for any $t,x,\xi$ as above, where $b_0:=\sup_{t\in [0,T]}b^{(0)}$.
It is then easy to check that
\begin{eqnarray*}
\begin{array}{ll}
|D^{\beta}b_j(t,x)|\le \kappa_1(1+|x|^2)^q, &|\beta|=2,3,\\[1mm]
|D^{\beta}c(t,x)|\le \kappa_2(1+|x|^2)^r,   &|\beta|=1,2,3,
\end{array}
\end{eqnarray*}
for some positive constants $\kappa_1$ and $\kappa_2$, any $(t,x)\in
[0,T]\times\R^N$ and any $j=1,\ldots,N$. Hence, we can take
$r(t,x):=\kappa_1(1+|x|^2)^q$ and
$\varrho(t,x):=\kappa_2(1+|x|^2)^r$ in \eqref{cond-deriv-b} and
\eqref{cond-deriv-c}. Condition \eqref{cond-funz} then reads as
follows:
\begin{eqnarray*}
b_0(1+|x|^2)^q+L_1\kappa_1(1+|x|^2)^q+L_2\kappa_2^2(1+|x|^2)^{2r}\le
L_3\nu_0(1+|x|^2)^p,
\end{eqnarray*}
for any $t\in [0,T]$ and any $x\in\R^N$. This inequality is clearly satisfied by suitable constants
$L_1,L_2,L_3$ by Hypothesis \ref{ipos-3}(iii). Finally, taking $\varphi(x)=1+|x|^2$ for any $x\in\R^N$,
we get
\begin{align}
({\mathscr A}\varphi)(t,x)&=2{\rm Tr}(Q(t,x))+2b^{(0)}(t)|x|^2(1+|x|^2)^q+(c_0(t,x)-|x|^{2r})(1+|x|^2)\nonumber\\
&\le
2\|Q^{(0)}\|_{\infty}(1+|x|^2)^p\hs{1}+\hs{1}2b_0|x|^2(1+|x|^2)^q\hs{1}+\hs{1}(\|c^{(0)}\|_{\infty}-|x|^{2r})(1+|x|^2),
\label{last}
\end{align}
where $\|Q^{(0)}\|_{\infty}=\sup_{(t,x)\in
[0,T]\times\R^N}\|Q^{(0)}(t,x)\|$. Due to Hypothesis
\ref{ipos-3}(iii), we can estimate the last side of \eqref{last}
from above by $\kappa_3+\|c^{(0)}\|_{\infty}(1+|x|^2)$ for any
$(t,x)\in [0,T]\times\R^N$ and some positive constant $\kappa_3$.
Hence, Hypothesis \ref{ipos-1}(vii) is satisfied with
$\lambda=\|c^{(0)}\|_{\infty}+\kappa_3$.

\section*{Acknowledgments}
The author wishes to thank the anonymous referees for the careful reading of the paper.


\end{document}